\documentclass[reqno]{amsart}
\usepackage{amssymb}
\usepackage{amsmath}
\usepackage{amsthm}
\usepackage[usenames]{color}
\usepackage{graphicx}
\usepackage{cite}
\usepackage{bbm}
\allowdisplaybreaks[4]

\usepackage[colorlinks,linkcolor=black,anchorcolor=black,citecolor=black]{hyperref}
\usepackage[margin=1.1in]{geometry} 
\usepackage{marginnote}

\usepackage{color}

\newtheorem{thm}{Theorem}[section]

\newtheorem{lem}[thm]{Lemma}
\newtheorem{prop}[thm]{Proposition}
\newtheorem{rem}{Remark}[section]

\numberwithin{equation}{section}

\newcommand{\ka}{\kappa}

\newcommand{\vertiii}[1]{{\left\vert\kern-0.25ex\left\vert\kern-0.25ex\left\vert #1
		\right\vert\kern-0.25ex\right\vert\kern-0.25ex\right\vert}}

\makeatletter
\@namedef{subjclassname@2020}{%
	\textup{2020} Mathematics Subject Classification}
\makeatother

\begin{document}

\title[Modified compensators for the incompressible Euler--VFP system] 
{Modified compensating functions for the incompressible Euler--Vlasov--Fokker--Planck system: Global classical solutions and pointwise-in-space decay}

\author[J. Ni] {Jinkai Ni$^*$}\thanks{$^*$Corresponding author. E-mail address: jinkaini123@gmail.com.}
\address[JKN]{School  of Mathematics, Nanjing University, Nanjing 
 210093, P. R. China}
\email{jinkaini123@gmail.com}

\begin{abstract}
We consider the Cauchy problem for the incompressible Euler-Vlasov-Fokker-Planck (Euler-VFP) system in the whole space
\(\mathbb R^3\) near the global Maxwellian equilibrium. The
Fokker-Planck operator and the particle-fluid drag dissipate the relative momentum but do not separately control the common particle-fluid momentum; in Fourier variables, this degeneracy occurs in the transverse momentum components.
To recover the missing coercivity, we augment the classical four-moment compensator with a finite-rank skew-adjoint correction constructed from second-order Hermite modes. Combined with
the cancellation between the kinetic and fluid drag terms and the
incompressibility constraint, the resulting compensated Fourier energy yields a unique global classical solution for sufficiently small initial
data $(u_0,f_0)\in H^N\times L_v^2(H^N)$, with $N\geq 4$.
The high-order energy argument involves only spatial derivatives of the
kinetic perturbation and requires no mixed \(x\)-\(v\) derivative
estimates. 
We further construct a positive-order Lyapunov functional and
establish the decay rate \((1+t)^{-1/2}\) for all positive-order spatial
derivatives in the \(L^2\)-norm and for the corresponding
pointwise-in-space norms, without any additional \(L^1\) integrability or low-frequency assumption on the initial data. Although no uniform
algebraic decay rate is asserted for the zero-order energy of
\((u,f)\), the directly dissipative variables \(u-J(f)\) and
\(\{\mathbf I-\mathbf P_0\}f\) decay in the \(L^2\)-norm at the same rate, where $
J(f)=\int_{\mathbb R^3}v\sqrt M f\,{\rm d}v$
denotes the particle momentum and \(\mathbf P_0\) is the orthogonal
projection onto
\(\operatorname{span}\{\sqrt M,v_1\sqrt M,v_2\sqrt M,v_3\sqrt M\}\).
To the best of our knowledge, these positive-order and zero-order decay estimates have not previously been established for the incompressible Euler-VFP system.

\end{abstract}

\date{\today}
	

\maketitle
\thispagestyle{empty}

\section{Introduction and main results}

\subsection{Introduction}
Fluid–particle interaction models provide a mesoscopic description of dispersed two-phase flows. The dispersed phase is represented by a particle distribution in phase space, whereas the surrounding dense phase is described by macroscopic fluid variables.
Typical examples include sprays and aerosols \cite{BBBDLLT-irma-2005,BD-JHDE-2006}, sedimentation of solid grains and suspension flows \cite{BBKT-SIAM-2003,CP-SIAM-1983}, fuel droplets in combustion theory \cite{Williams-1958,Williams-1985}, evaporating droplets in spray processes \cite{RM-1952-I,RM-1952-II}, and biosprays in medicine \cite{BBJM-esaim-2005,Moussa-PhD-2009}.
At the modelling level, solid non-deformable particles suspended in a carrier fluid are described by a kinetic equation for their distribution in phase space. The action of the carrier flow on the particles is encoded through a frictional relaxation towards the local fluid velocity, whereas microscopic fluctuations of particle velocities are represented by a Fokker–Planck diffusion; see, for instance, \cite{MV-MMMAS-2007}.

In this paper, we investigate the Cauchy problem for the following
three-dimensional incompressible inviscid fluid-particle system
considered in \cite{CDM-KRM-2011}:
\begin{equation}\label{I1}
\left\{
\begin{aligned}
&\partial_t u+u\cdot\nabla_x u+\nabla_x P
   =\int_{\mathbb R^3}(v-u)F\,{\rm d}v,\\
&\nabla_x\cdot u=0,\\
&\partial_t F+v\cdot\nabla_x F
   =\nabla_v\cdot [(v-u)F+\nabla_vF ],
\end{aligned}
\right.
\end{equation}
with   initial data
\begin{align}\label{I1-1}
(u,F)|_{t=0}=(u_0(x),F_0(x,v)).
\end{align}
Here, $t\geq0$, $x\in\mathbb{R}^3$, and $v\in\mathbb{R}^3$. The  function $F=F(t,x,v)\geq0$ is the particle distribution in phase space, whereas $u=u(t,x)\in\mathbb{R}^3$ and $P=P(t,x)$ denote the velocity field and pressure of the incompressible carrier fluid. 
The coupling is governed by the relative velocity \(v-u\): its
\(F\)-weighted velocity moment produces the friction force in the fluid
equation, while the same quantity enters the drift term of the
Vlasov-Fokker-Planck equation; see
\cite{Williams-1985,CP-SIAM-1983}. The initial data are also assumed to satisfy
\begin{align}\label{G1.3}  
\nabla_x\cdot u_0=0,\qquad F_0(x,v)\geq0.  
\end{align}

Our aim is to investigate the global-in-time dynamics of solutions to
\eqref{I1} near the equilibrium \((F,u)=(M,0)\). In particular, we establish the global existence of small classical solutions and describe their large-time behavior, including the decay of positive spatial derivatives and the zero-order \(L^2\) decay of the directly dissipative variables, without imposing any additional \(L^1\) or low-frequency assumption on the initial data.
After normalizing the physical parameters, we take the global Maxwellian to be
\begin{align*}
 M=M(v)={(2\pi)^{-\frac{3}{2}}}e^{-\frac{|v|^{2}}{2}}.   
\end{align*}
To characterize such solutions in the vicinity of this equilibrium, we introduce the perturbation $f=f(t,x,v)$ through
\begin{align*}
F=M+\sqrt{M}f.
\end{align*}
The corresponding perturbation problem \eqref{I1}--\eqref{I1-1} is then written as
\begin{equation}\label{A1}
\left\{
\begin{aligned}
&\partial_t u+u\cdot\nabla_x u+\nabla_x P+u
   =\int_{\mathbb R^3}v\sqrt{M}f\,{\rm d}v-u\int_{\mathbb R^3}\sqrt{M}f\,{\rm d}v,\\
&\nabla_x\cdot u=0,\\
&\partial_t f+v\cdot\nabla_x f+u\cdot\nabla_{v}f-\frac{1}{2}u\cdot v f-u\cdot v\sqrt{M}
   = \mathcal{L}f,
\end{aligned}
\right.
\end{equation}
with initial data
\begin{align}\label{A-1}
(u,f)|_{t=0}=(u_0(x),f_0(x,v))=\Big(u_0(x), \frac{F_0(x,v)-M}{\sqrt{M}} \Big),
\end{align}
where $\mathcal{L}$ is the linearized Fokker–Planck operator given by 
\begin{align}\label{G1.5}
\mathcal{L}f=\frac{1}{\sqrt{M}}\nabla_{v}\cdot
\Big[M\nabla_{v}\Big(\frac{f}{\sqrt{M}}\Big)\Big]=\Delta_{v}f-\frac{|v|^{2}}{4}f+\frac{3}{2}f,
\end{align}
which has the {\it one dimensional} null space $\mathcal{N}(\mathcal{L})={\rm span}\{\sqrt{M}\}$.

Many early studies of kinetic–fluid models concerned viscous couplings. Hamdache \cite{Hk-jjiam-1998} established the global existence and large time behavior of weak solutions to the Vlasov-Stokes system. For the incompressible Navier-Stokes-Vlasov equations, Goudon, Jabin, and Vasseur \cite{GJV-IUMJ-2004,GJV-IUMJ-2004-2} investigated hydrodynamic limits in the light- and fine-particle regimes, while Boudin, Desvillettes, Grandmont, and Moussa \cite{BD-jneq-2009} constructed global weak solutions on the torus $\mathbb T^3$. 
Related global weak-solution results were obtained by Yu \cite{Yc-JMPA-2013} and, for a time-dependent domain, by Boudin, Grandmont, and Moussa \cite{BGM-JDE-2017}.
When Fokker-Planck diffusion is included, the incompressible Navier-Stokes-VFP system admits a  Maxwellian equilibrium. Goudon, He, Moussa, and Zhang \cite{GHMZ-2010-SIMA} established the global existence and exponential stability of strong solutions near equilibrium in the periodic domain. 
Chae, Kang, and Lee \cite{CKL-JDE-2011} obtained global weak and classical solutions for the related Navier-Stokes-VFP system. For the corresponding compressible models, Mellet and Vasseur \cite{MV-MMMAS-2007,MV-CMP-2008} proved the existence of global finite-energy weak solutions and investigated their asymptotic behavior.
Subsequently, global classical solutions near equilibrium and their large time behavior were studied in \cite{CKL-JHDE-2013,LMW-SIAM-2017}.

Removing the fluid viscosity term $-\mu \Delta u$ leads to Euler-kinetic systems. In the absence of the particle diffusion term $-\Delta_{v}F$, Baranger and Desvillettes \cite{BD-JHDE-2006} established the local existence of classical solutions for a compressible Euler-Vlasov system arising in spray dynamics. When the Fokker-Planck term is retained, Carrillo and Goudon \cite{CG-CPDE-2006} analyzed the stability and asymptotic behavior of a fluid-particle interaction model in several physical regimes. Duan and Liu \cite{DL-krm-2013}  proved the global well-posedness of small classical solutions to the compressible Euler-VFP system. They also established algebraic convergence in the whole space $\mathbb{R}^3$ and exponential convergence in the periodic domain $\mathbb{T}^3$. 
Further results on compressible fluid–particle models can be found in \cite{Ww-CMS-2024,LNW-2025-preprint,Mu-Wang-20,LNW-arxiv-2026,JP-jde-2017,Jp-2020,LWW-ARMA-2022} and references cited therein.

We now shift our focus to incompressible fluid–particle systems. Concerning the inhomogeneous Navier–Stokes–Vlasov system characterized by a density-dependent drag force, Wang and Yu \cite{WY-JDE-2015} proved the global existence of weak solutions in a bounded domain. Hydrodynamic limits for inhomogeneous incompressible kinetic–fluid systems were investigated by Su and Yao \cite{SY-JDE-2020} for the Navier–Stokes–VFP equations and by Su, Wu, Yao, and Zhang \cite{SWYZ-JDE-2024} for the corresponding Navier–Stokes–Vlasov   equations. Recently, Jiang, Li, and the author of this paper \cite{JLN-2025-JMP} obtained the global existence and large time decay of classical solutions to the incompressible inhomogeneous Navier–Stokes–VFP system in the whole space $\mathbb{R}^3$.
For the homogeneous incompressible inviscid system \eqref{I1} considered in this paper, Carrillo, Duan, and Moussa \cite{CDM-KRM-2011} established global classical solutions and investigated their large time behavior. In particular, under an additional $L^1$ assumption on the initial data, they obtained the decay rate
\begin{align*}   
 \| f(t
) \|_{L_v^2(H ^3)} + \|u(t)\|_{H ^3} \leq C_{\delta} (1+t)^{-\frac{3}{4}+\delta}, \quad\text{for}\quad \delta>0, 
\end{align*}  
where $C_{\delta}$ depends only on $\delta$ and may diverge as $\delta\rightarrow0^+$.

In this paper, we revisit system \eqref{I1} and establish the global well-posedness of small classical solutions by means of a modified
compensating-function method. Originating in Kawashima's work
\cite{Kawashima1990}, this approach was formulated in a kinetic setting
in \cite{Glassey} and further developed by Yang and Yu
\cite{YangYu2010}. To our knowledge, the method has not previously been adapted to the momentum-exchange structure of particle-fluid coupling models. A compensator tailored to the relative momentum between the particle and fluid phases is therefore constructed. Besides providing
the uniform a priori estimates for \((u,f)\) required by the global theory, the resulting Lyapunov structure yields the decay rate
\((1+t)^{-1/2}\) for positive-order spatial derivatives in the
\(L^2\)-norm and for the corresponding pointwise-in-space norms, without
any additional \(L^1\) or low-frequency assumption on the initial data.
Moreover, although no uniform algebraic decay rate is available for the complete zero-order energy, the directly dissipative variables
\(u-J(f)\) and \(\{\mathbf I-\mathbf P_0\}f\) decay in the zero-order
\(L^2\)-norm at the same rate. Such zero-order decay of the directly
dissipative variables was not established in
\cite{CDM-KRM-2011}.

\subsection{The compensating-function method and the Euler-VFP structure}

We first recall the compensating mechanism used in our analysis. It is closely
related to partially dissipative hyperbolic systems of balance laws, where
relaxation acts directly on only part of the unknowns. Shizuta and Kawashima
\cite{SK-HMJ-1985} introduced the coupling condition through which dissipation
is transferred to the remaining components. Further developments concerning
dissipative structures, global existence, and large-time behavior can be found
in \cite{KawashimaYong2004,BianchiniHanouzetNatalini2007,
BeauchardZuazua2011}.

Consider the following linear partially dissipative hyperbolic system:
\begin{align*}
\partial_t U+\sum_{j=1}^3 A_j\partial_{x_j}U+\mathbf L U=0,
\end{align*}
where the matrices \(A_j\) are symmetric and \(\mathbf L\) is symmetric and
nonnegative definite, with a nontrivial kernel. For
\begin{align*}
\omega=\frac{\xi}{|\xi|}\in\mathbb S^2,
\qquad
A(\omega)=\sum_{j=1}^3\omega_jA_j,
\end{align*}
the Shizuta-Kawashima condition requires that no nonzero vector in
\(\ker\mathbf L\) be an eigenvector of \(A(\omega)\). In other words,
transport couples each undamped mode to a dissipative one.
Under this condition, one can construct a skew-symmetric matrix
\(K(\omega)\), depending smoothly on \(\omega\), such that
\begin{align*}
\frac12\big(
K(\omega)A(\omega)
+\bigl[K(\omega)A(\omega)\bigr]^{\top}
\big)
+\mathbf L
\end{align*}
is uniformly positive definite for \(\omega\in\mathbb S^2\). More precisely,
there exist constants \(C_0,c_0>0\), independent of
\(\omega\in\mathbb S^2\) and \(Z\in\mathbb C^n\), such that
\begin{align*}
\operatorname{Re}
\bigl(
K(\omega)A(\omega)Z,Z
\bigr)_{\mathbb C^n}
+
C_0
\bigl(
\mathbf L Z,Z
\bigr)_{\mathbb C^n}
\geq
c_0|Z|^2.
\end{align*}
Here, \((\cdot,\cdot)_{\mathbb C^n}\) denotes the standard Hermitian inner
product on \(\mathbb C^n\). Although \(\mathbf L\) is degenerate, coupling
with transport transfers the available dissipation to modes in
\(\ker\mathbf L\). With a suitable Fourier weight, one obtains the
dissipation symbol $\frac{|\xi|^2}{1+|\xi|^2}$,
which describes the frequency-dependent behavior of solutions:
low-frequency modes decay diffusively, whereas high-frequency modes are
uniformly damped.

In kinetic equations, \(K(\omega)\) is replaced by a bounded skew-adjoint
operator on \(L_v^2\), while \(A(\omega)\) becomes multiplication by
\(v\cdot\omega\). Yang and Yu \cite{YangYu2010} developed such a construction
for the Vlasov-Maxwell-Fokker-Planck (VMFP) system on the four-moment space
\begin{align}\label{G1.8WW}
\mathcal W
=
\operatorname{span}
\big\{
\sqrt M,\,
v_1\sqrt M,\,
v_2\sqrt M,\,
v_3\sqrt M
\big\}.
\end{align}
Their argument combines the kinetic compensator with Maxwell’s equations and Gauss’s law, which relate particle moments to the electromagnetic fields and provide additional control over the field variables.
No analogous field structure is available for the Euler-VFP system
\eqref{A1}. Instead, particle momentum is coupled directly to the
incompressible fluid velocity through the drag force. To describe the
resulting degeneracy, fix \(\xi\neq0\) and set
\(\omega=\frac{\xi}{|\xi|}\).
For \(\widehat f=\widehat f(t,\xi,v)\), define its density and momentum moments by
\begin{align*}
a(\widehat f)
=
\langle\widehat f,\sqrt M\rangle_v,
\qquad
J(\widehat f)
=
\langle\widehat f,v\sqrt M\rangle_v.
\end{align*}
The divergence-free condition \eqref{A1}$_2$ gives
\begin{align*}
\omega\cdot\widehat u=0.
\end{align*}
Accordingly, \(J(\widehat f)\) is decomposed into its longitudinal and
transverse parts:
\begin{align*}
J(\widehat f)=
\bigl(J(\widehat f)\cdot\omega\bigr)\omega
+
\bigl({\rm Id}-\omega\otimes\omega\bigr)J(\widehat f)
&=:J_\parallel(\widehat f)+J_\perp(\widehat f).
\end{align*}
Let \(\mathbf P\) denote the orthogonal projection onto
\(\mathcal N(\mathcal L)\). Any function \(f=f(t,x,v)\) admits the unique
decomposition
\begin{align*}
f=\mathbf Pf+\{\mathbf I-\mathbf P\}f,
\end{align*}
where \(\mathbf Pf\) and \(\{\mathbf I-\mathbf P\}f\) are referred to as
the macroscopic and microscopic components, respectively. In Fourier
variables,
\begin{align*}
\mathbf P\widehat f
=
a(\widehat f)\sqrt M.
\end{align*}
The standard coercivity of the Fokker-Planck operator is expressed in terms
of this projection:
\begin{align}\label{intro-FP-coercivity}
-\langle\mathcal L\widehat f,\widehat f\rangle_v
=
-\bigl\langle
\mathcal L\{\mathbf I-\mathbf P\}\widehat f,
\{\mathbf I-\mathbf P\}\widehat f
\bigr\rangle_v\geq
\lambda_0
\bigl|\{\mathbf I-\mathbf P\}\widehat f\bigr|_\nu^2,
\end{align}
for some constant \(\lambda_0>0\).

To isolate the particle momentum, we introduce the four-moment
projection \(\mathbf P_0\) onto \(\mathcal W\):
\begin{align*}
\mathbf P_0\widehat f
=
a(\widehat f)\sqrt M
+
J(\widehat f)\cdot v\sqrt M.
\end{align*}
Unlike \(\mathbf P\), the projection \(\mathbf P_0\) is not the null-space
projection associated with the coercivity of \(\mathcal L\);  it is introduced to separate the momentum modes in the coupled particle–fluid energy. Since
\begin{align*}
\mathcal L(v_i\sqrt M)=-v_i\sqrt M,
\qquad 1\leq i\leq3,
\end{align*}
one has
\begin{align*}
-\langle\mathcal L\widehat f,\widehat f\rangle_v=
|J(\widehat f)|^2
-
\bigl\langle
\mathcal L\{\mathbf I-\mathbf P_0\}\widehat f,
\{\mathbf I-\mathbf P_0\}\widehat f
\bigr\rangle_v\geq
|J(\widehat f)|^2
+
\lambda_1
\bigl|\{\mathbf I-\mathbf P_0\}\widehat f\bigr|_\nu^2,
\end{align*}
for some constant \(\lambda_1>0\). Hence, $\mathcal{L}$
remains coercive on 
\(\{\mathbf I-\mathbf P_0\}\widehat f\),
although \(\mathbf P_0\) is introduced primarily to separate the momentum
modes.
Consequently, the linearized particle-fluid energy has the exact
dissipation form
\begin{align}\label{intro-relative-dissipation}
\mathcal D(\widehat f,\widehat u)
&:=
-\langle\mathcal L\widehat f,\widehat f\rangle_v
+|\widehat u|^2
-2\operatorname{Re}
\bigl(\widehat u\cdot\overline{J(\widehat f)}\bigr)
\nonumber\\
&=
|\widehat u-J(\widehat f)|^2
-
\bigl\langle
\mathcal L\{\mathbf I-\mathbf P_0\}\widehat f,
\{\mathbf I-\mathbf P_0\}\widehat f
\bigr\rangle_v
\nonumber\\
&=
|\widehat u-J_\perp(\widehat f)|^2
+|J_\parallel(\widehat f)|^2
-
\bigl\langle
\mathcal L\{\mathbf I-\mathbf P_0\}\widehat f,
\{\mathbf I-\mathbf P_0\}\widehat f
\bigr\rangle_v
\nonumber\\
&\geq
|\widehat u-J_\perp(\widehat f)|^2
+|J_\parallel(\widehat f)|^2
+
\lambda_1
\bigl|\{\mathbf I-\mathbf P_0\}\widehat f\bigr|_\nu^2,
\end{align}
where the third equality follows from
\(\omega\cdot\widehat u=0\).

Although \eqref{intro-FP-coercivity} controls the full microscopic component
\(\{\mathbf I-\mathbf P\}\widehat f\), the particle-fluid cross term changes
the zero-order dissipative structure. More precisely,
\eqref{intro-relative-dissipation} controls the relative momentum
\(\widehat u-J_\perp(\widehat f)\), the longitudinal component
\(J_\parallel(\widehat f)\), and the four-moment remainder
\(\{\mathbf I-\mathbf P_0\}\widehat f\), but does not separately control \(\widehat u\) and \(J_\perp(\widehat f)\) at zero order. This degeneracy is
consistent with conservation of total particle-fluid momentum.
Unlike the VMFP setting in \cite{YangYu2010}, where Maxwell equations assist
the four-moment compensator, the four-moment construction alone does not remove the transverse degeneracy in the Euler-VFP system. We therefore add a
correction coupling \(J_\perp(\widehat f)\) to dissipative second-order
Hermite modes. Since
\begin{align*}
\{\mathbf I-\mathbf P\}\widehat f
=
J(\widehat f)\cdot v\sqrt M
+
\{\mathbf I-\mathbf P_0\}\widehat f,
\end{align*}
controlling the momentum modes together with the coercivity of
\(\{\mathbf I-\mathbf P_0\}\widehat f\) recovers the full microscopic
component associated with \(\mathbf P\). The resulting skew-adjoint operator
\(S(\omega)\) satisfies
\begin{align}\label{G1.8}
\operatorname{Re}
\bigl\langle
S(\omega)(v\cdot\omega)\widehat f,\widehat f
\bigr\rangle_v
+C_0\mathcal D(\widehat f,\widehat u)
\geq
c_0\Bigl(
|a(\widehat f)|^2
+|\widehat u|^2
+\bigl|\{\mathbf I-\mathbf P\}\widehat f\bigr|_\nu^2
\Bigr),
\end{align}
where \(C_0,c_0>0\) are independent of
\(\omega\in\mathbb S^2\), \(\widehat f\), and \(\widehat u\). The construction of
\(S(\omega)\) and the proof of \eqref{G1.8} are given in Section 2.


We now state the main results.

\subsection{Main results}

\begin{thm}\label{Th1}
Let \( N \geq 4 \) be an integer, and suppose that condition \eqref{G1.3} holds. Assume that
the initial data \((u_0,f_0)\in H^N\times L_v^2(H^N)\) satisfy
\begin{align*}
 \|f_0\|_{L_v^2(H^N)}+\|  u_0\|_{H^N} \leq \varepsilon,  
\end{align*}
for some sufficiently small constant \(\varepsilon>0\).
Then the Cauchy problem \eqref{A1}--\eqref{A-1} admits a unique global
classical solution \((u,f)\) such that
\(F=M+\sqrt{M}f\geq 0\),  
\begin{align*}
f\in C\big([0,\infty);L_v^2(H^N)\big),
\qquad
u\in C\big([0,\infty);H^N\big),  
\end{align*}
and
\begin{align}\label{TA2}
  \|f(t)\|_{L_v^2(H^N)}^2+\| u(t)\|_{H^N}^2 + \int_0^t D_{N}(\tau)\,{\rm d}\tau  \leq C\big(\| u_0\|_{H^N}^2+\|f_0\|_{L_v^2(H^N)}^2\big).
\end{align}
Moreover, the positive-order spatial derivatives satisfy
\begin{align}\label{TA3a}
&\sum_{1\leq|\alpha|\leq N-1}
\left\{
\|\partial_x^\alpha f(t)\|_{L_{x,v}^2}
+
\|\partial_x^\alpha u(t)\|_{L^2}
\right\}\leq
C(1+t)^{-\frac12}
\left(
\|f_0\|_{L_v^2(H^N)}
+
\|u_0\|_{H^N}
\right),
\end{align}
for any $t\geq 0$.
Furthermore, the corresponding pointwise-in-space estimate holds:
\begin{align}\label{TA3}
&\sup_{x\in \mathbb R^3}\bigg\{ \sum_{|\alpha|\leq N-3}\bigg(\int_{\mathbb R^3}|\partial^\alpha_{x} f(t,x,v)|^2\,{\rm d}v 
      \bigg)^\frac{1}{2} +\sum_{|\alpha|\leq N-3} |\partial^\alpha u(t,x)|         \bigg\}\nonumber\\
&\quad      \leq C(1+t)^{-\frac{1}{2}}\left(
\|f_0\|_{L_v^2(H^N)}
+
\|u_0\|_{H^N}
\right),   
\end{align}
for any $t\geq 0$. Here, \(C>0\) is independent of \(t\).
\end{thm}

\begin{rem}
The dissipation functional in Theorem~\ref{Th1} is defined by
\begin{align}\label{TB1}
\mathcal D_N(t):={}&\,
\|u(t)-J(f(t))\|_{H^N}^2
+\|\{\mathbf I-\mathbf P_0\}f(t)\|_\nu^2
\nonumber\\
&+\sum_{1\leq|\alpha|\leq N}
\|\partial_x^\alpha\{\mathbf I-\mathbf P\}f(t)\|_\nu^2
+\|\nabla_x\mathbf P f(t)\|_{L_v^2(H^{N-1})}^2
+\|\nabla_xu(t)\|_{H^{N-1}}^2.
\end{align}
At zero spatial-derivative order, the kinetic part of \eqref{TB1} contains only the four-moment remainder
\(\{\mathbf I-\mathbf P_0\}f\), which is weaker than the full microscopic
dissipation since
\begin{align*}
\|\{\mathbf I-\mathbf P_0\}f\|_\nu
\lesssim
\|\{\mathbf I-\mathbf P\}f\|_\nu.    
\end{align*} 
The missing momentum mode is controlled through the relative momentum \(u-J(f)\).
The projection denoted by \(\mathbf P\) in \cite{CDM-KRM-2011,LNW-arxiv-2026} coincides with \(\mathbf P_0\) in the present notation.
\end{rem}

\begin{rem}
The argument used here differs from the macro-micro energy method employed in
\cite{CDM-KRM-2011}. Our proof is based on a modified compensating-function
method in Fourier space.
The global well-posedness result established in
\cite{CDM-KRM-2011} is complemented here by the pointwise decay estimate
\eqref{TA3}, obtained without any \(L^1\) or other low-frequency
assumption on the initial data.
\end{rem}

\begin{rem}
Compared with the Sobolev framework used in \cite{YangYu2010},
our a priori estimates require only spatial derivatives of the kinetic perturbation, namely \(f_0\in L_v^2(H^N)\); no mixed \(x\)-\(v\) derivatives enter the global energy argument.
\end{rem}

\begin{thm}\label{Th2}
Under the assumptions of Theorem \ref{Th1}, the global classical
solution $(u,f)$ to system \eqref{A1}--\eqref{A-1} satisfies
\begin{align}\label{TA4}
&\|u(t)-J(f(t))\|_{L^2}
+\big\|\{\mathbf I-\mathbf P_0\}f(t)\big\|_{L_{x,v}^2}
\leq
C(1+t)^{-\frac{1}{2}}
\left(
\|f_0\|_{L_v^2(H^N)}+\|u_0\|_{H^N}
\right),
\end{align}
for any $t\geq 0$.  Here, \(C>0\) is independent of \(t\).
\end{thm}

\begin{rem}
For both the incompressible and compressible Euler-VFP systems, the algebraic convergence of the solution obtained in
\cite{CDM-KRM-2011,DL-krm-2013} relies on additional assumptions on the initial data. Theorem \ref{Th2} reveals a different relaxation mechanism:
without any \(L^1\) or low-frequency assumption, $
u-J(f)$ and $\{\mathbf I-\mathbf P_0\}f$
still decay in the zero-order \(L^2\) norm at the rate
\((1+t)^{-1/2}\). This is consistent with the undamped common
particle-fluid momentum at zero frequency, since that mode is canceled
in \(u-J(f)\) and is absent from
\(\{\mathbf I-\mathbf P_0\}f\).
\end{rem}

 \subsection{Strategies in our proofs}

The argument combines the Fourier compensating framework developed in
\cite{Kawashima1990,Glassey,YangYu2010} with the particle-fluid energy
structure of system \eqref{A1}--\eqref{A-1}. For each \(\xi\neq0\), we
add to the direct Fourier energy the correction
\begin{align*}
-\frac{\kappa}{2}
\frac{|\xi|}{1+|\xi|^2}
\big\langle
\mathrm iS(\omega)\widehat f,\widehat f
\big\rangle_v,
\quad \text{with}\quad
\omega=\frac{\xi}{|\xi|}.
\end{align*}
Upon differentiating this term in time, the kinetic transport operator
produces
\begin{align*}
\kappa\frac{|\xi|^2}{1+|\xi|^2}
\operatorname{Re}
\big\langle
S(\omega)(v\cdot\omega)\widehat f,\widehat f
\big\rangle_v.
\end{align*}
Combining this contribution with the direct particle-fluid dissipation
\eqref{intro-relative-dissipation} and applying the modified coercivity
estimate \eqref{G1.8}, we arrive at the pointwise Fourier inequality
\eqref{eq:pointwise-Fourier-energy}.
More precisely, the direct energy identity provides zero-order dissipation for \(u-J(f)\) and
\(\{\mathbf I-\mathbf P_0\}f\), whereas the compensating term recovers the missing frequency-weighted control of $a(f)$, $u$ and \(\{\mathbf I-\mathbf P\}f\).
In addition, the resulting pointwise estimate yields the
frequency-weighted control of
\(\{\mathbf I-\mathbf P\}f\). Multiplying
\eqref{eq:pointwise-Fourier-energy} by
\((1+|\xi|^2)^N\) and integrating over \(\xi\), we obtain
\eqref{eq:Fourier-energy-ineq}, namely,
\begin{align}\label{NJKG3}
\frac{{\rm d}}{{\rm d}t}\mathcal E_N(t)
+
\lambda_2 \mathcal D_N(t)
\leq
\mathcal N_{1,N}(t)
+
C\mathcal N_{2,N}(t).
\end{align}
Furthermore, \eqref{eq:energy-equivalence} gives
\begin{align*}
\mathcal E_N(t)
\sim
\|f(t)\|_{L_v^2(H ^N)}^2+\|u(t)\|_{H ^N}^2,
\end{align*}
and the dissipative terms obtained after integration coincide with
\(\mathcal D_N(t)\) defined in \eqref{TB1}.

The nonlinear analysis focuses on the two source terms
\(\mathcal N_{1,N}(t)\) and \(\mathcal N_{2,N}(t)\) defined in
\eqref{eq:N1}--\eqref{eq:N2}. The main difficulty in estimating
\(\mathcal N_{1,N}(t)\) is that the kinetic nonlinearity and the
nonlinear drag term in the fluid equation must be treated simultaneously.
Applying the same spatial derivatives to both terms and expanding them with the same Leibniz coefficients allows the density contributions to be reorganized in terms of the relative momentum \(u-J(f)\). The remaining
kinetic contributions are paired with
\(\{\mathbf I-\mathbf P_0\}f\) and are therefore controlled by the
zero-order microscopic dissipation. For the Euler transport term
\(u\cdot\nabla_xu\), the principal contribution vanishes by the
divergence-free condition \(\nabla_x\cdot u=0\), leaving only
commutator terms. Applying the Sobolev product and commutator estimates
in Lemmas~\ref{L2.1}--\ref{L2.2} to these terms gives
(see \eqref{eq:N1-bound})
\begin{align}\label{NJKG1} 
|\mathcal N_{1,N}(t)|
\leq
C\sqrt{\mathcal E_N(t)}\,\mathcal D_N(t).
\end{align}

After integration by parts in \(v\), the finite velocity moments
appearing in \(\mathcal N_{2,N}(t)\) reduce to products involving
\(u\), \(a(f)\), and \(J(f)\). This yields
(see \eqref{eq:N2-bound})
\begin{align}\label{NJKG2}
\mathcal N_{2,N}(t)
\leq
C\mathcal E_N(t) \mathcal D_N(t).
\end{align}
Substitution of \eqref{NJKG1} and \eqref{NJKG2} into
\eqref{NJKG3}  gives the closed energy inequality
(see \eqref{eq:closed-energy})
\begin{align*} 
\frac{{\rm d}}{{\rm d}t}\mathcal E_N(t)
+
\lambda_4\mathcal D_N(t)
\leq
C\sqrt{\mathcal E_N(t)}\,\mathcal D_N(t).
\end{align*}
A continuity argument then yields the uniform estimate \eqref{TA2}.
Only spatial derivatives are applied to the kinetic equation. Since
\(\partial_x^\alpha\) commutes with \(\mathcal L\), the term
\(-\langle \mathcal L\partial_x^\alpha f,\partial_x^\alpha f\rangle_v\)
 controls
\(\partial_x^\alpha\{\mathbf I-\mathbf P\}f\) in the velocity-dissipation norm, including
\(\nabla_v\partial_x^\alpha \{\mathbf I-\mathbf P\} f\).  
Hence, no mixed \(x\)-\(v\) derivative estimates are required.

For the decay estimate, we follow the positive-order energy method of
\cite{YangYu2010}. The functional
\(\mathcal E_N^{(1)}(t)\), defined in \eqref{eq:positive-energy}, is
equivalent to the squared \(L^2\) norms of the positive spatial derivatives of
\((u,f)\). Theorem~\ref{thm:positive-energy} gives
\begin{align*}
\frac{{\rm d}}{{\rm d}t}\mathcal E_N^{(1)}(t)
+
\lambda_5\mathcal D_N^{(1)}(t)
\leq
C\mathcal D_N(t)\mathcal E_N^{(1)}(t)
+
C\bigl(\mathcal E_N^{(1)}(t)\bigr)^3.
\end{align*}
Since
\(
\mathcal E_N^{(1)}(t)
\lesssim
\mathcal D_N(t)
\)
by \eqref{eq:Eplus-by-D}, the global energy estimate implies that
\(\mathcal E_N^{(1)}\) is integrable over \((0,\infty)\). Setting
\(y(t)=\mathcal E_N^{(1)}(t)\), we obtain
\begin{align*}
 y'(t)
\leq
C\bigl(\mathcal D_N(t)+y(t)^2\bigr)y(t).   
\end{align*}
Here, $y\in L^1(0,\infty)$ and \(\mathcal D_N(t)+y(t)^2\in L^1(0,\infty)\) because $y$ is uniformly bounded.
Applying the
Gr\"{o}nwall-type lemma in \cite{Dk-MZ-1992,YangYu2010}, we obtain
\begin{align*}
\mathcal E_N^{(1)}(t)
\leq
C(1+t)^{-1}\mathcal E_N(0),
\end{align*}
for any $t\geq 0$.
The equivalence of \(\mathcal E_N^{(1)}(t)\) with the positive-order
energy then yields the \(L^2\) decay estimate
\eqref{eq:positive-L2-decay}. The pointwise estimate \eqref{TA3} follows
from the Sobolev inequality stated in Lemma~\ref{L2.1}. The 
$L^2$ decay result \eqref{TA3a}
is formulated only for positive spatial derivatives because the common
particle-fluid momentum mode remains undamped at zero spatial frequency.

To prove Theorem~\ref{Th2}, we use the zero-order damping of the directly
dissipative variables \(u-J(f)\) and
\(\{\mathbf I-\mathbf P_0\}f\). Taking the \(v\sqrt M\)-moment of the
kinetic equation \eqref{A1}$_3$ and subtracting the resulting equation from the
Leray-projected fluid equation \eqref{A1}$_1$ yields the estimate
\eqref{eq:sec5-relative-energy} for \(u-J(f)\). Applying
the operator \(\{\mathbf I-\mathbf P_0\}\) to \eqref{A1}$_3$ and using the spectral gap
of the Fokker-Planck operator gives
\eqref{eq:sec5-microscopic-energy}. Combining these two estimates with
the positive-order decay and the uniform energy bound, we obtain
\eqref{eq:sec5-forcing-decay}, namely,
\begin{align*}
\frac{{\rm d}}{{\rm d}t}
\big\{
\|u-J(f)\|_{L^2}^2
+\big\|\{\mathbf I-\mathbf P_0\}f\big\|_{L_{x,v}^2}^2
\big\}
+\lambda_6
\big\{
\|u-J(f)\|_{L^2}^2
+\big\|\{\mathbf I-\mathbf P_0\}f\big\|_{L_{x,v}^2}^2
\big\}
\leq C(1+t)^{-1}\mathcal E_N(0).
\end{align*}
Applying Gr\"{o}nwall's inequality, we obtain \eqref{TA4}. This reveals
a  zero-order relaxation phenomenon: although the complete
zero-order energy admits no uniform algebraic decay rate, the directly
dissipative variables still decay at the rate
\((1+t)^{-1/2}\) without any additional low-frequency assumption.

\subsection{Outline of this paper}
The remainder of the paper is organized as follows. In Section~2, we
introduce the notation, construct the modified compensating operator,
establish its coercivity, and derive the compensated Fourier energy
inequality. Several Sobolev, product, and commutator estimates used in the
subsequent analysis are collected at the end of the section. In
Section~3, we estimate the nonlinear terms and prove the global existence
of classical solutions to the Cauchy problem
\eqref{A1}--\eqref{A-1}. Section~4 is devoted to the positive-order
\(L^2\) decay and the corresponding pointwise decay estimate
\eqref{TA3}, thereby completing the proof of
Theorem~\ref{Th1}. In Section~5, we establish the zero-order \(L^2\)
decay of the directly dissipative variables \(u-J(f)\) and
\(\{\mathbf I-\mathbf P_0\}f\), and complete the proof of
Theorem~\ref{Th2}.

\section{Preliminaries}
In this section, we introduce the   notations,  the compensating function method for the Euler-VFP system \eqref{A1}, along with a set of useful lemmas that will be frequently employed throughout this paper.

\subsection{Notations}

Throughout the paper, \(C>0\) denotes a generic constant independent of
time, while \(C_\eta>0\) denotes a constant depending only on
\(\eta>0\); both may vary from line to line. For nonnegative quantities
\(A\) and \(B\), we write \(A\lesssim B\) if \(A\leq CB\).
The notation \(A\sim B\) means that both \(A\lesssim B\) and
\(B\lesssim A\) hold. For a Banach space \(X\) and \(h_1,h_2\in X\), we set
\begin{align*}
\|(h_1,h_2)\|_X:=\|h_1\|_X+\|h_2\|_X.
\end{align*}
For an integrable function \(h=h(x)\), its Fourier transform with respect to
the spatial variable is defined by
\begin{align*}
\widehat h(\xi)
=
\mathcal F_x h(\xi)
:=
\int_{\mathbb R^3}
e^{-ix\cdot\xi}h(x)\,{\rm d}x,
\qquad
x\cdot\xi=\sum_{j=1}^3x_j\xi_j,
\end{align*}
where \(i=\sqrt{-1}\). For \(h=h(x,v)\),
\(\widehat h=\widehat h(\xi,v)\) always denotes the Fourier transform with
respect to \(x\) only.

We use \(\langle\cdot,\cdot\rangle_v\) and
\((\cdot,\cdot)_{x,v}\) to denote the standard complex inner products in
\(L^2(\mathbb R_v^3)\) and
\(L^2(\mathbb R_x^3\times\mathbb R_v^3)\), respectively.
All complex inner products are taken to be linear in the first entry.
The corresponding norms are denoted by
\(|\cdot|_{L_v^2}\) and \(\|\cdot\|_{L_{x,v}^2}\), respectively.
For functions depending only on \(x\), we write \(\|\cdot\|_{L^2}\). 

To describe the dissipation generated by the Fokker--Planck operator
\(\mathcal L\), we set
\begin{align*}
\nu(v):=1+|v|^2.
\end{align*}
For \(h=h(v)\), define
\begin{align*}
|h|_\nu^2
:=
\int_{\mathbb R^3}
\left(
|\nabla_vh(v)|^2+\nu(v)|h(v)|^2
\right)\,{\rm d}v,
\end{align*}
whereas for \(h=h(x,v)\), we use
\begin{align*}
\|h\|_\nu^2
:=
\int_{\mathbb R^3}\!\!\int_{\mathbb R^3}
\left(
|\nabla_vh(x,v)|^2+\nu(v)|h(x,v)|^2
\right)\,{\rm d}v{\rm d}x.
\end{align*}

Let \(\alpha=(\alpha_1,\alpha_2,\alpha_3)\) be a multi-index. We denote
\begin{align*}
\partial_x^\alpha
:=
\partial_{x_1}^{\alpha_1}
\partial_{x_2}^{\alpha_2}
\partial_{x_3}^{\alpha_3},
\qquad
|\alpha|:=\alpha_1+\alpha_2+\alpha_3.
\end{align*}
For brevity, for each \(i=1,2,3\), we write \(\partial_i\) instead of
\(\partial_{x_i}\).   The Sobolev norms involving only spatial derivatives are
defined by
\begin{align*}
\|h\|_{L_v^2(H^m)}^2:=
\sum_{|\alpha|\leq m}
\|\partial_x^\alpha h\|_{L_{x,v}^2}^2, \qquad
\|h\|_{H^m}^2
:=
\sum_{|\alpha|\leq m}
\|\partial_x^\alpha h\|_{L^2}^2.
\end{align*}

For later use, we set
\begin{align*}
e_1=\sqrt M,
\qquad
e_{i+1}=v_i\sqrt M,
\quad 1\leq i\leq3.
\end{align*}
The family \(\{e_j\}_{j=1}^4\) is an orthonormal basis of the
four-moment space \(\mathcal W\) defined in \eqref{G1.8WW}.

For a phase-space function $h=h(x,v)$, we set
\begin{align*}
a(  h)
=
\langle  h,\sqrt M\rangle_v,
\qquad
J(  h)
=
\langle  h,v\sqrt M\rangle_v,  
\end{align*}
and define 
\begin{align*}
\mathbf{P}h=a(h)\sqrt{M},\qquad \mathbf{P}_0h=   a(h)\sqrt{M}+ J(h)\cdot v\sqrt{M}.
\end{align*}
The same notation is used after Fourier transformation in $x$.

\subsection{The compensating operator}

Following the construction of the compensating function for the Boltzmann equation in \cite{Glassey,Kawashima1990} and its kinetic formulation in
\cite[Lemma 2.1]{YangYu2010}, we construct a compensating operator adapted
to the Euler-VFP system \eqref{A1}. The construction combines the classical four-moment component with a finite-rank correction based on second-order Hermite profiles.

  Let $\omega = (\omega_1, \omega_2, \omega_3)^T \in \mathbb{S}^2$.
The compression of multiplication by \(v\cdot\omega\) to the four-moment
space \(\mathcal W\) is represented by
\begin{align*}
V(\omega)
&:=
\left\{
\left\langle
(v\cdot\omega)e_\ell,e_k
\right\rangle_v
\right\}_{k,\ell=1}^4
 =
\begin{pmatrix}
0&\omega_1&\omega_2&\omega_3\\
\omega_1&0&0&0\\
\omega_2&0&0&0\\
\omega_3&0&0&0
\end{pmatrix}.
\end{align*}
Following \cite[Lemma 2.1]{YangYu2010}, we introduce the following skew-symmetric matrix:
\begin{align*}
R(\omega)
=
\begin{pmatrix}
0&\omega_1&\omega_2&\omega_3\\
-\omega_1&0&0&0\\
-\omega_2&0&0&0\\
-\omega_3&0&0&0
\end{pmatrix}.
\end{align*}
Since \(|\omega| = 1\), direct multiplication gives
\begin{align*}
R(\omega)V(\omega)
&=
\begin{pmatrix}
1&0&0&0\\
0&-\omega_1^2&-\omega_1\omega_2&-\omega_1\omega_3\\
0&-\omega_2\omega_1&-\omega_2^2&-\omega_2\omega_3\\
0&-\omega_3\omega_1&-\omega_3\omega_2&-\omega_3^2
\end{pmatrix}=
\begin{pmatrix}
1&0\\
0&-\omega\otimes\omega
\end{pmatrix},
\end{align*}
which implies that
\begin{align}\label{eq:RV}
\operatorname{Re}
\bigl(
R(\omega)V(\omega)W,W
\bigr)_{\mathbb C^4}
=
|W_1|^2
-
|\omega_1W_2+\omega_2W_3+\omega_3W_4|^2,
\end{align}
for any \(W=(W_1,W_2,W_3,W_4)^T\in\mathbb C^4\).
Set $R(\omega)=\big(r_{k\ell}(\omega)\big)_{1\leq k,\ell\leq4}$. Then, we define
\begin{align}\label{eq:S1}
S^{(1)}(\omega)f
=
\gamma_1
\sum_{k,\ell=1}^4
r_{k\ell}(\omega)
\langle f,e_\ell\rangle_v e_k,
\end{align}
where \(\gamma_1>0\) will be chosen later. In particular, if we set
\begin{align*}
W(f):=\big(a(f),J_1(f),J_2(f),J_3(f)\big)^T,\end{align*} 
then it can be deduced from \eqref{eq:RV} that
\begin{align*}
\operatorname{Re}
\bigl(
R(\omega)V(\omega)W(f),W(f)
\bigr)_{\mathbb C^4}
=
|a(f)|^2-|J_\parallel(f)|^2.
\end{align*}
Thus, \(S^{(1)}(\omega)\) contributes positively to the density mode and negatively to the longitudinal momentum mode, while leaving \(J_\perp(f)\) uncontrolled. To compensate for this missing
transverse component, we define the profiles \(\phi_m\) by
\begin{align}\label{eq:phi}
\phi_m(v;\omega)
=
(v\cdot\omega)(\Pi_\omega v)_m\sqrt M,
\qquad
1\leq m\leq3,
\end{align}
where
\[
\Pi_\omega:={\rm Id}-\omega\otimes\omega.
\]
We further define
\begin{align}\label{eq:S2}
S^{(2)}(\omega)f
=
\gamma_2\sum_{m=1}^3
\left\{
\langle f,\phi_m\rangle_v e_{m+1}
-
\langle f,e_{m+1}\rangle_v\phi_m
\right\},
\end{align}
for some constant \(\gamma_2>0\). The full compensating operator is defined by
\begin{align}\label{eq:S}
S(\omega)
=
S^{(1)}(\omega)+S^{(2)}(\omega).
\end{align}

We next establish the basic properties of the profiles \(\phi_m\).

\begin{lem}\label{lem:phi-S}
For any \(\omega\in\mathbb S^2\) and \(1\leq m,j\leq3\), the profiles
\(\phi_m\) defined in \eqref{eq:phi} have the following properties:
\begin{enumerate}
\item[\rm (i)]
Each \(\phi_m\) belongs to \(\mathcal W^\perp\) and is an eigenfunction of
the Fokker-Planck operator corresponding to the eigenvalue \(-2\), namely,
\begin{align}\label{eq:phi-properties}
\phi_m\perp\mathcal W,
\qquad
\mathcal L\phi_m=-2\phi_m.
\end{align}

\item[\rm (ii)]
The following moment identities hold:
\begin{align}\label{eq:phi-transport}
\left\langle
e_{j+1},(v\cdot\omega)\phi_m
\right\rangle_v
=
(\Pi_\omega)_{mj},
\qquad
\sum_{m=1}^3\omega_m\phi_m=0.
\end{align}
\end{enumerate}
\end{lem}

\begin{proof}
We first prove \eqref{eq:phi-properties}.  The Gaussian moment identities give
\begin{align*}
\int_{\mathbb R^3}v_pv_qM\,{\rm d}v=
\delta_{pq},\qquad
\int_{\mathbb R^3}v_pv_qv_rv_sM\,{\rm d}v=
\delta_{pq}\delta_{rs}
+\delta_{pr}\delta_{qs}
+\delta_{ps}\delta_{qr},
\end{align*}
 which yields  
\begin{align}\label{G2.8}
\langle\phi_m,e_1\rangle_v=
\sum_{p,q=1}^3
\omega_p(\Pi_\omega)_{mq}
\int_{\mathbb R^3}v_pv_qM\,{\rm d}v
=
\sum_{p=1}^3
\omega_p(\Pi_\omega)_{mp}
=0.
\end{align}
Moreover, the integrand in \eqref{G2.9} is an odd polynomial times the centered Maxwellian $M$; hence 
\begin{align}\label{G2.9}
\langle\phi_m,e_{j+1}\rangle_v
=
\int_{\mathbb R^3}
v_j(v\cdot\omega)(\Pi_\omega v)_mM\,{\rm d}v=0,
\end{align}
for $j=1,2,3$.
It follows from \eqref{G2.8} and \eqref{G2.9} that $\phi_m\perp\mathcal W$.

For \(1\leq p,q\leq3\), we define $H_{pq}
:=
(v_pv_q-\delta_{pq})\sqrt M$.
A direct calculation from \eqref{G1.5} shows that
\begin{align}\label{eq:L-Hpq}
\mathcal LH_{pq}=-2H_{pq}.
\end{align}
Moreover, since
\begin{align*}
\sum_{p,q=1}^3
\omega_p(\Pi_\omega)_{mq}\delta_{pq}
=
\sum_{p=1}^3
\omega_p(\Pi_\omega)_{mp}
=
(\Pi_\omega\omega)_m
=0,
\end{align*}
the profile \(\phi_m\) can be written as
\begin{align}
\phi_m
&=
\sum_{p,q=1}^3
\omega_p(\Pi_\omega)_{mq}v_pv_q\sqrt M
\nonumber\\
&=
\sum_{p,q=1}^3
\omega_p(\Pi_\omega)_{mq}
(v_pv_q-\delta_{pq})\sqrt M
\nonumber\\
&=
\sum_{p,q=1}^3
\omega_p(\Pi_\omega)_{mq}H_{pq}.
\label{eq:phi-Hpq}
\end{align}
Notice that \(\mathcal L\) is linear and acts only on the velocity variable \(v\). Thus, the coefficients \(\omega_p(\Pi_\omega)_{mq}\) may be taken outside the operator. It follows from \eqref{eq:L-Hpq} and \eqref{eq:phi-Hpq} that
\begin{align*}
\mathcal L\phi_m=
\sum_{p,q=1}^3
\omega_p(\Pi_\omega)_{mq}\mathcal LH_{pq}=
-2
\sum_{p,q=1}^3
\omega_p(\Pi_\omega)_{mq}H_{pq}=
-2\phi_m,
\end{align*}
which proves the second identity in \eqref{eq:phi-properties}.

We next prove \eqref{eq:phi-transport}. By direct calculation, one gets
\begin{align*}
\left\langle
e_{j+1},(v\cdot\omega)\phi_m
\right\rangle_v
&=
\sum_{p,q,r=1}^3
\omega_p\omega_r(\Pi_\omega)_{mq}
\left(
\delta_{jp}\delta_{rq}
+\delta_{jr}\delta_{pq}
+\delta_{jq}\delta_{pr}
\right)\\
&=
2\omega_j(\Pi_\omega\omega)_m
+
|\omega|^2(\Pi_\omega)_{mj}\\
&=
(\Pi_\omega)_{mj},
\end{align*}
where we have used $\Pi_\omega\omega=0$.
Moreover, it follows directly from \eqref{eq:phi} that
\begin{align*}
\sum_{m=1}^3\omega_m\phi_m
&=
(v\cdot\omega)
\bigl(\omega\cdot\Pi_\omega v\bigr)\sqrt M
=0,
\end{align*}
which proves \eqref{eq:phi-transport} and completes the proof of Lemma \ref{lem:phi-S}.
\end{proof}

With the help of Lemma \ref{lem:phi-S}, we can further obtain the properties of the compensating function $S(\omega)$ in \eqref{eq:S}.

\begin{lem}\label{lem:S-properties}
Let \(\gamma_1,\gamma_2>0\) be fixed. The operator \(S(\omega)\) defined by  \eqref{eq:S} has the following properties:
\begin{enumerate}
\item[\rm (i)]
The map \(\omega\mapsto S(\omega)\) is \(C^\infty\) on
\(\mathbb S^2\), with values in the space of bounded linear operators on
\(L^2(\mathbb R_v^3)\). Moreover, there exists a constant \(C>0\),
independent of \(\omega\in\mathbb S^2\), such that
\begin{align}\label{eq:S-bounded}
\|S(\omega)f\|_{L_v^2}
\leq
C\|f\|_{L_v^2},
 \qquad
f\in L^2(\mathbb R_v^3),
\end{align}
and
\begin{align}\label{eq:S-odd}
S(-\omega)=-S(\omega).
\end{align}

\item[\rm (ii)]
For any \(\omega\in\mathbb S^2\), the operator \(S(\omega)\) is
skew-adjoint in \(L^2(\mathbb R_v^3)\), while
\(\mathrm iS(\omega)\) is self-adjoint. More precisely,
\begin{align}\label{eq:S-adjoint}
S(\omega)^*=-S(\omega),
\qquad
\big(\mathrm iS(\omega)\big)^*
=
\mathrm iS(\omega),
\end{align}
where \(S(\omega)^*\) denotes the adjoint operator characterized by
\begin{align*}
\langle S(\omega)f,h\rangle_v
=
\langle f,S(\omega)^*h\rangle_v,
\qquad
f,h\in L^2(\mathbb R_v^3).
\end{align*}

\item[\rm (iii)]
For any \(z\in\mathbb C^3\) satisfying \(z\cdot\omega=0\), one has
\begin{align}\label{eq:S-fluid}
S^{(1)}(\omega)(z\cdot v\sqrt M)
=0,
\qquad
S(\omega)(z\cdot v\sqrt M)
=
-\gamma_2\sum_{m=1}^3z_m\phi_m.
\end{align}
\end{enumerate}
\end{lem}

\begin{proof}
We initially prove part {\rm (i)}. Given that the entries \(r_{k\ell}(\omega)\) and the profiles \(\phi_m(\cdot;\omega)\) exhibit a polynomial dependence on \(\omega\), one can conclude that
\begin{align*} 
\sup_{\omega\in\mathbb S^2}
\Big\{
\max_{1\leq k,\ell\leq4}|r_{k\ell}(\omega)|
+
\sum_{m=1}^3
\|\phi_m(\cdot;\omega)\|_{L_v^2}
\Big\}
<\infty.
\end{align*}
On the one hand, it should be noted that \(\{e_k\}_{k = 1}^4\) is an orthonormal basis of \(\mathcal W\). Then, based on \eqref{eq:S1} and the Cauchy-Schwarz inequality, it can be deduced that
\begin{align}\label{G2.16}
\|S^{(1)}(\omega)f\|_{L_v^2}
\leq
\gamma_1
\sum_{k,\ell=1}^4
|r_{k\ell}(\omega)|
\,|\langle f,e_\ell\rangle_v|
\,\|e_k\|_{L_v^2}\leq
C\gamma_1\|f\|_{L_v^2}.
\end{align}
On the other hand, \eqref{eq:S2} gives
\begin{align}\label{G2.17}
\|S^{(2)}(\omega)f\|_{L_v^2}
\leq
\gamma_2\sum_{m=1}^3
\Big\{
|\langle f,\phi_m\rangle_v|
\,\|e_{m+1}\|_{L_v^2}
+
|\langle f,e_{m+1}\rangle_v|
\,\|\phi_m\|_{L_v^2}
\Big\}\leq
C\gamma_2\|f\|_{L_v^2},
\end{align}
where \(C>0\) is independent of
\(\omega\in\mathbb S^2\). Combining the estimates \eqref{G2.16} and \eqref{G2.17} leads to  \eqref{eq:S-bounded}.

By the definitions of \(R(\omega)\), \(\Pi_\omega\), and \(\phi_m\), one has
\begin{align*}
R(-\omega) =-R(\omega),\qquad
\Pi_{-\omega} =\Pi_\omega,\qquad
\phi_m(v;-\omega)=-\phi_m(v;\omega).
\end{align*}
It then follows directly from \eqref{eq:S1} and \eqref{eq:S2} that
\begin{align*}
S^{(1)}(-\omega)=-S^{(1)}(\omega),
\qquad
S^{(2)}(-\omega)=-S^{(2)}(\omega).
\end{align*}
Therefore,
\begin{align*}
S(-\omega)=-S(\omega),
\end{align*}
which proves \eqref{eq:S-odd}.

We next prove part {\rm (ii)}. Since $R(\omega)^T=-R(\omega)$, a direct calculation based on \eqref{eq:S1} yields
\begin{align*}
\langle S^{(1)}(\omega)f,h\rangle_v=-\langle f,S^{(1)}(\omega)h\rangle_v.
\end{align*} 
Therefore, \(S^{(1)}(\omega)\) is skew-adjoint.
Each summand in \(S^{(2)}(\omega)\) is an operator of rank at most two of the form
\begin{align*}
\mathcal A_{e,\phi}f
:=
\langle f,\phi\rangle_v e
-
\langle f,e\rangle_v\phi.
\end{align*}
Since the inner product \(\langle\cdot,\cdot\rangle_v\) is linear in its first argument, for any \(f,h\in L^2(\mathbb R_v^3)\), we obtain
\begin{align*}
\langle\mathcal A_{e,\phi}f,h\rangle_v
=
-\langle f,\mathcal A_{e,\phi}h\rangle_v,
\end{align*}
which leads to
\begin{align*}
\mathcal A_{e,\phi}^*
=
-\mathcal A_{e,\phi}.
\end{align*}
Thus, each summand in \(S^{(2)}(\omega)\) is skew-adjoint, and \(S^{(2)}(\omega)\) itself is also skew-adjoint. Considering the skew-adjointness of \(S^{(1)}(\omega)\), we can conclude that
\begin{align*}
S(\omega)^*=-S(\omega).
\end{align*}
Consequently,
\begin{align*}
\big(\mathrm iS(\omega)\big)^*
=
-\mathrm iS(\omega)^*
=
\mathrm iS(\omega),
\end{align*}
which proves \eqref{eq:S-adjoint}.

It remains to verify part {\rm (iii)}.
For notational simplicity,  we set
\begin{align*}
f_z(v)
:=
z\cdot v\sqrt M
=
\sum_{m=1}^3 z_m e_{m+1}.
\end{align*}
where
the four-moment vector associated with \(f_z\) is
\begin{align*}
\left(
\langle f_z,e_1\rangle_v,
\langle f_z,e_2\rangle_v,
\langle f_z,e_3\rangle_v,
\langle f_z,e_4\rangle_v
\right)^T
=
\begin{pmatrix}
0\\ z
\end{pmatrix}.
\end{align*}
According to the definition of \(R(\omega)\), we have
\begin{align*}
R(\omega)
\begin{pmatrix}
0\\ z
\end{pmatrix}
=
\begin{pmatrix}
\omega\cdot z\\ 0
\end{pmatrix}.
\end{align*}
Thus, under the condition \(z\cdot\omega = 0\), we directly arrive at \(S^{(1)}(\omega)f_z = 0\) from \eqref{eq:S1}.

Moreover,  \(f_z\in\mathcal W\), whereas
\(\phi_m\perp\mathcal W\)  by 
Lemma \ref{lem:phi-S}. Therefore, we have
\begin{align*}
\langle f_z,\phi_m\rangle_v=0,
\qquad
\langle f_z,e_{m+1}\rangle_v=z_m.
\end{align*}
Substituting  the above identities into \eqref{eq:S2}, one obtains
\begin{align*}
S^{(2)}(\omega)f_z
=
-\gamma_2\sum_{m=1}^3z_m\phi_m,
\end{align*}
which together with \(S^{(1)}(\omega)f_z=0\)  yields
\begin{align*}
S(\omega)f_z
=
-\gamma_2\sum_{m=1}^3z_m\phi_m.
\end{align*}
Then,   \eqref{eq:S-fluid} follows, and we complete the proof of Lemma \ref{lem:S-properties}.
\end{proof}

We next prove the modified coercivity estimate stated in \eqref{G1.8}.

\begin{lem}\label{lem:modified-coercivity}
The constants \(\gamma_1,\gamma_2>0\) in
\eqref{eq:S1} and \eqref{eq:S2} can be chosen such that there exist
constants \(C_0,c_0>0\), independent of
\(\omega\in\mathbb S^2\), with the following property: for any
\(f\in L_v^2\) satisfying \(|f|_\nu<\infty\) and every
\(z\in\mathbb C^3\) satisfying \(z\cdot\omega=0\), one has
\begin{align}\label{eq:modified-coercivity}
\operatorname{Re}
\left\langle
S(\omega)(v\cdot\omega)f,f
\right\rangle_v
+
C_0\mathcal D(f,z)\geq
c_0
\big\{
|a(f)|^2
+
|z|^2
+
\big|
\{\mathbf I-\mathbf P\}f
\big|_\nu^2
\big\}.
\end{align}
Here \(\mathcal D(f,z)\) denotes the quadratic form introduced in
\eqref{intro-relative-dissipation}, with
\((\widehat f,\widehat u)\) replaced by \((f,z)\).
\end{lem}

\begin{proof}
Thanks to the definition of \(\mathbf P_0\), one derives
\begin{align}\label{eq:P0-decomp}
f
=
a(f)\sqrt M
+
J(f)\cdot v\sqrt M
+
\{\mathbf I-\mathbf P_0\}f.
\end{align}
For $1 \leq j \leq 3$, we define $q_j$ by  
\begin{align*}
q_j
:=
\left\langle
(v\cdot\omega)
\{\mathbf I-\mathbf P_0\}f,
e_{j+1}
\right\rangle_v,
\qquad
q=(q_1,q_2,q_3).
\end{align*}
It is easy to check that
\begin{align*}
\left\langle
(v\cdot\omega)
\{\mathbf I-\mathbf P_0\}f,
e_1
\right\rangle_v
=
\left\langle
\{\mathbf I-\mathbf P_0\}f,
(v\cdot\omega)e_1
\right\rangle_v=0,
\end{align*}
where we used the following relations:
\begin{align*}
(v\cdot\omega)e_1\in\mathcal W,
\qquad
\{\mathbf I-\mathbf P_0\}f\perp\mathcal W.
\end{align*}
Below, we utilize the decomposition $J(f)=J_\parallel(f)+J_\perp(f)$ introduced in Subsection 1.2.

By leveraging \eqref{eq:RV}, \eqref{eq:S1}, and the decomposition
\eqref{eq:P0-decomp}, one obtains
\begin{align}
\operatorname{Re}
\langle
S^{(1)}(\omega)(v\cdot\omega)f,f
\rangle_v=
\gamma_1
\left(
|a(f)|^2-|J_\parallel(f)|^2
\right)
+
\gamma_1\operatorname{Re}
\big\{
(\omega\cdot q)\overline{a(f)}
\big\}.
\label{eq:S1-exact}
\end{align}
For each \(j\), utilizing the Cauchy-Schwarz inequality gives
\begin{align*}
|q_j|\leq
\bigl\|
\{\mathbf I-\mathbf P_0\}f
\bigr\|_{L_v^2}
\,
\bigl\|
(v\cdot\omega)e_{j+1}
\bigr\|_{L_v^2}\leq
C
\bigl|
\{\mathbf I-\mathbf P_0\}f
\bigr|_\nu,
\end{align*}
where \(C>0\) is independent of
\(\omega\in\mathbb S^2\), which gives rise to
\begin{align*}
|q|
\lesssim
\bigl|
\{\mathbf I-\mathbf P_0\}f
\bigr|_\nu.
\end{align*}
Applying Young's inequality to the last term in
\eqref{eq:S1-exact}, we obtain
\begin{align}\label{eq:S1-lower}
\operatorname{Re}
 \langle
S^{(1)}(\omega)(v\cdot\omega)f,f
 \rangle_v\geq
\frac{\gamma_1}{2}|a(f)|^2
-
\gamma_1|J_\parallel(f)|^2
-
C\gamma_1
\big|
\{\mathbf I-\mathbf P_0\}f
\big|_\nu^2.
\end{align}

Next, we consider the estimate of \(S^{(2)}(\omega)\). It follows from
\eqref{eq:P0-decomp} and Lemma \ref{lem:phi-S} that
\begin{align}
\left\langle
(v\cdot\omega)f,\phi_m
\right\rangle_v
&=
\bigl(J_\perp(f)\bigr)_m
+
\left\langle
(v\cdot\omega)\{\mathbf I-\mathbf P_0\}f,\phi_m
\right\rangle_v,
\label{eq:transport-phi}\\
\left\langle
(v\cdot\omega)f,e_{m+1}
\right\rangle_v
&=
\omega_ma(f)
+
\left\langle
(v\cdot\omega)\{\mathbf I-\mathbf P_0\}f,e_{m+1}
\right\rangle_v.
\label{eq:transport-e}
\end{align}
Moreover,  it holds that
\begin{align}\label{eq:omega-orthogonality}
\sum_{m=1}^3
\omega_m
\left\langle
\{\mathbf I-\mathbf P_0\}f,\phi_m
\right\rangle_v
=0,\qquad
\sum_{m=1}^3
\omega_m
\left\langle
(v\cdot\omega)\{\mathbf I-\mathbf P_0\}f,\phi_m
\right\rangle_v
=0,
\end{align}
where we used
$\sum\limits_{m=1}^3\omega_m\phi_m=0$.
From \eqref{eq:omega-orthogonality}, we find that
\begin{align*}
&\sum_{m=1}^3
\left\langle
(v\cdot\omega)\{\mathbf I-\mathbf P_0\}f,\phi_m
\right\rangle_v
\overline{J_m(f)}=
\sum_{m=1}^3
\left\langle
(v\cdot\omega)\{\mathbf I-\mathbf P_0\}f,\phi_m
\right\rangle_v
\overline{\bigl(J_\perp(f)\bigr)_m}.
\end{align*}
By inserting the identities \eqref{eq:transport-phi}--\eqref{eq:omega-orthogonality}
into \eqref{eq:S2}, one obtains that
\begin{align*}
&\operatorname{Re}
\langle
S^{(2)}(\omega)(v\cdot\omega)f,f
\rangle_v-\gamma_2|J_\perp(f)|^2\nonumber\\
&\qquad=
\gamma_2\operatorname{Re}
\sum_{m=1}^3
\Big\{
\left\langle
(v\cdot\omega)\{\mathbf I-\mathbf P_0\}f,\phi_m
\right\rangle_v
\overline{\bigl(J_\perp(f)\bigr)_m}
-
\left\langle
(v\cdot\omega)\{\mathbf I-\mathbf P_0\}f,e_{m+1}
\right\rangle_v
\overline{
\left\langle
\{\mathbf I-\mathbf P_0\}f,\phi_m
\right\rangle_v
}
\Big\},
\end{align*}
which together with Young's inequality implies that
\begin{align}\label{eq:S2-lower}
 \operatorname{Re}
 \langle
S^{(2)}(\omega)(v\cdot\omega)f,f
 \rangle_v\geq
\frac{\gamma_2}{2}|J_\perp(f)|^2
-
C\gamma_2
\big|
\{\mathbf I-\mathbf P_0\}f
\big|_\nu^2.
\end{align}

Adding \eqref{eq:S1-lower} and \eqref{eq:S2-lower}, and then using
\eqref{intro-relative-dissipation}, we deduce that
\begin{align}\label{eq:combined-coercivity}
\operatorname{Re}
\left\langle
S(\omega)(v\cdot\omega)f,f
\right\rangle_v
+
C_0\mathcal D(f,z)\geq&\,
\frac{\gamma_1}{2}|a(f)|^2
+
\frac{\gamma_2}{2}|J_\perp(f)|^2
+
(C_0-\gamma_1)|J_\parallel(f)|^2
+
C_0|z-J_\perp(f)|^2\nonumber\\
&+
\bigl(
C_0\lambda_1-C(\gamma_1+\gamma_2)
\bigr)
\big|
\{\mathbf I-\mathbf P_0\}f
\big|_\nu^2.
\end{align}
Choosing \(C_0>0\) in \eqref{eq:combined-coercivity} sufficiently large, we further derive that
\begin{align*}
|a(f)|^2
+
|z|^2
+
\big|
\{\mathbf I-\mathbf P\}f
\big|_\nu^2\lesssim&\,
 |a(f)|^2
+
|J_\parallel(f)|^2
+
|J_\perp(f)|^2
+
|z-J_\perp(f)|^2
+
\bigl|
\{\mathbf I-\mathbf P_0\}f
\bigr|_\nu^2\\
\lesssim&\, \operatorname{Re}
\left\langle
S(\omega)(v\cdot\omega)f,f
\right\rangle_v
+
C_0\mathcal D(f,z) .
\end{align*}
Consequently, the desired estimate \eqref{eq:modified-coercivity} follows.
\end{proof}

We now apply the preceding compensating function to the Euler-VFP system \eqref{A1}. By applying the Leray projection \(\mathbb P\) to the first equation in \eqref{A1}, we  define the nonlinear terms $g$ and $G$ as follows:
\begin{equation}\label{eq:gG}
\left\{
\begin{aligned}
g
&:=
-u\cdot\nabla_vf
+\frac12(u\cdot v)f,
\\
G
&:=
-\mathbb P(u\cdot\nabla_xu)
-\mathbb P\big (a(f)u\big ).
\end{aligned}
\right.
\end{equation}
Taking the Fourier transform with respect to \(x\) in \eqref{A1}, for \(\omega = \xi/|\xi|\), we obtain that
\begin{equation}\label{eq:Fourier-system}
\left\{
\begin{aligned}
&\partial_t\widehat f
+\mathrm i|\xi|(v\cdot\omega)\widehat f
-\mathcal L\widehat f
=
\widehat u\cdot v\sqrt M+\widehat g,
\\
&\omega\cdot\widehat u=0,\\
&\partial_t\widehat u+\widehat u
=
\Pi_\omega J(\widehat f)+\widehat G.
\end{aligned}
\right.
\end{equation}
For \(0<\kappa\ll1\), we introduce the compensated energy
\(\mathcal E_N(t)\) by
\begin{align*} 
\mathcal E_N(t)
:=\frac12
\int_{\mathbb R^3}
(1+|\xi|^2)^N
\Big[
&
 \big(
\|\widehat f\|_{L_v^2}^2+|\widehat u|^2
 \big)-
\frac{\ka|\xi|}{1+|\xi|^2}
 \langle
\mathrm iS(\omega)\widehat f,\widehat f
 \rangle_v
\Big]
\,{\rm d}\xi.
\end{align*}
We collect the nonlinear terms arising from the basic energy identity in
\begin{align}\label{eq:N1}
\mathcal N_{1,N}(t)
:=
\int_{\mathbb R^3}
(1+|\xi|^2)^N
\Big\{
\operatorname{Re}
\langle\widehat g,\widehat f\rangle_v+
\operatorname{Re}
\big(
\widehat G\cdot\overline{\widehat u}
\big)
\Big\}\,{\rm d}\xi.
\end{align}
In addition, the interaction term generates finite-dimensional moments of
\(\widehat g\), which are measured by
\begin{align}\label{eq:N2}
\mathcal N_{2,N}(t)
:=&\,
\sum_{k=1}^4
\int_{\mathbb R^3}
(1+|\xi|^2)^{N-1}
|\langle\widehat g,e_k\rangle_v|^2
\,{\rm d}\xi
\nonumber\\
&+
\sum_{1\leq i\leq j\leq3}
\int_{\mathbb R^3}
(1+|\xi|^2)^{N-1}
 \big|
 \big\langle
\widehat g,
(v_iv_j-\delta_{ij})\sqrt M
 \big\rangle_v
 \big|^2
\,{\rm d}\xi.
\end{align}

The compensated energy \(\mathcal E_N(t)\) and the nonlinear quantities
\(\mathcal N_{1,N}(t)\) and \(\mathcal N_{2,N}(t)\) introduced above satisfy the
following Fourier energy estimate, which will be used in the nonlinear analysis.

\begin{prop}\label{prop:Fourier-energy}
For sufficiently small \(\kappa>0\), one has
\begin{align}\label{eq:energy-equivalence}
\mathcal E_N(t)
\sim
\|f(t)\|_{L_v^2(H^N)}^2
+
\|u(t)\|_{H^N}^2.
\end{align}
Moreover, there exists a constant \(\lambda_2>0\), independent of \(t\),
such that
\begin{align}\label{eq:Fourier-energy-ineq}
\frac{{\rm d}}{{\rm d}t}\mathcal E_N(t)
+
\lambda_2 \mathcal D_N(t)
\leq
\mathcal N_{1,N}(t)
+
C\mathcal N_{2,N}(t).
\end{align}
\end{prop}

\begin{proof}
Pairing \eqref{eq:Fourier-system}$_1$ and
\eqref{eq:Fourier-system}$_3$ with \(\widehat f\) and
\(\widehat u\), respectively, and taking real parts, we obtain
\begin{align}
\frac12\partial_t
\big(
\|\widehat f\|_{L_v^2}^2+|\widehat u|^2
\big)
+
\mathcal D(\widehat f,\widehat u)
=
\operatorname{Re}
\langle\widehat g,\widehat f\rangle_v
+
\operatorname{Re}
\big(
\widehat G\cdot\overline{\widehat u}
\big),
\label{eq:basic-Fourier}
\end{align}
where \(\mathcal D(\widehat f,\widehat u)\) is defined in
\eqref{intro-relative-dissipation}. Here, we have used the fact
\begin{align*}
\operatorname{Re}
\big\langle
\mathrm i|\xi|(v\cdot\omega)\widehat f,\widehat f
\big\rangle_v
=0,
\end{align*}
together with the constraint
\(\omega\cdot\widehat u=0\) in \eqref{eq:Fourier-system}, which implies
\begin{align*}
\operatorname{Re}
\big(
\Pi_\omega J(\widehat f)\cdot\overline{\widehat u}
\big)
=
\operatorname{Re}
\big(
J(\widehat f)\cdot\overline{\widehat u}
\big).
\end{align*}

For the interaction term, we apply
\(-\mathrm i|\xi|S(\omega)\) to  
\eqref{eq:Fourier-system}$_1$ and then take the \(L_v^2\) inner product with
\(\widehat f\). The self-adjointness of \(\mathrm iS(\omega)\) yields
\begin{align}\label{eq:interaction-identity}
-\frac{|\xi|}{2}
\partial_t
\big\langle
\mathrm iS(\omega)\widehat f,\widehat f
\big\rangle_v
+
|\xi|^2
\operatorname{Re}
\big\langle
S(\omega)(v\cdot\omega)\widehat f,
\widehat f
\big\rangle_v
=&\,
-|\xi|
\operatorname{Re}
\big\langle
\mathrm iS(\omega)\mathcal L\widehat f,
\widehat f
\big\rangle_v
-|\xi|
\operatorname{Re}
\big\langle
\mathrm iS(\omega)
(\widehat u\cdot v\sqrt M),
\widehat f
\big\rangle_v\nonumber\\
&-|\xi|
\operatorname{Re}
\big\langle
\mathrm iS(\omega)\widehat g,
\widehat f
\big\rangle_v.
\end{align}
 By using Lemma \ref{lem:S-properties}, \(\mathrm iS(\omega)\) is self-adjoint in \(L_v^2\). Hence, \(\big\langle\mathrm iS(\omega)\widehat f,\widehat f\big\rangle_v\) is real-valued.
 We estimate  the terms on the right-hand side of \eqref{eq:interaction-identity} one by one.
From the relations $\mathcal Le_1=0$, $\mathcal Le_{m+1}=-e_{m+1}$ and \eqref{eq:phi-properties}, we have 
\begin{align*}
S^{(1)}(\omega)\mathcal L\widehat f
=&\,
-\gamma_1
\big(
\omega\cdot J(\widehat f)
\big)e_1,\\
S^{(2)}(\omega)\mathcal L\widehat f
=&\,
\gamma_2
\sum_{m=1}^3
\Big[
-2
\big\langle
\{\mathbf I-\mathbf P_0\}\widehat f,
\phi_m
\big\rangle_v
e_{m+1}
+
J_m(\widehat f)\phi_m
\Big],
\end{align*}
which together with \eqref{eq:S-fluid}  gives
\begin{align}
\big|
\big\langle
\mathrm iS(\omega)\mathcal L\widehat f,
\widehat f
\big\rangle_v
\big|
\lesssim&\,
|a(\widehat f)|
\,|J_\parallel(\widehat f)|+
\bigl|
\{\mathbf I-\mathbf P_0\}\widehat f
\bigr|_\nu
\,|J(\widehat f)|,
\label{eq:linear-error-L}\\
\big|
\big\langle
\mathrm iS(\omega)
(\widehat u\cdot v\sqrt M),
\widehat f
\big\rangle_v
\big|
\lesssim&\,
|\widehat u|
\bigl|
\{\mathbf I-\mathbf P_0\}\widehat f
\bigr|_\nu.
\label{eq:linear-error-u}
\end{align}
For the last term on the right-hand side of \eqref{eq:interaction-identity}, we infer that
\begin{align}
\big|
\big\langle
\mathrm iS(\omega)\widehat g,
\widehat f
\big\rangle_v
\big|\lesssim&\,
\bigg\{
\sum_{k=1}^4
|\langle\widehat g,e_k\rangle_v|^2
+
\sum_{1\leq i\leq j\leq3}
\big|
\big\langle
\widehat g,
(v_iv_j-\delta_{ij})\sqrt M
\big\rangle_v
\big|^2
\bigg\}^{\frac{1}{2}}
\nonumber\\
& \times
\Big\{
|a(\widehat f)|
+
|J(\widehat f)|
+
\big\|
\{\mathbf I-\mathbf P_0\}\widehat f
\big\|_{L_v^2}
\Big\}.
\label{eq:source-error}
\end{align}

By substituting \eqref{eq:linear-error-L}--\eqref{eq:source-error} into \eqref{eq:interaction-identity}, multiplying the resulting inequality by \(\kappa(1 + |\xi|^2)^{-1}\), and then applying Young's inequality, we can obtain, for any \(\eta > 0\), that
\begin{align}
&-\frac{\kappa}{2}
\frac{|\xi|}{1+|\xi|^2}
\partial_t
\big\langle
\mathrm iS(\omega)\widehat f,\widehat f
\big\rangle_v+
\kappa
\frac{|\xi|^2}{1+|\xi|^2}
\operatorname{Re}
\big\langle
S(\omega)(v\cdot\omega)\widehat f,\widehat f
\big\rangle_v
\nonumber\\
&\quad\leq 
\eta\kappa
\frac{|\xi|^2}{1+|\xi|^2}
\Big\{
|a(\widehat f)|^2
+
|\widehat u|^2
+
\big|
\{\mathbf I-\mathbf P\}\widehat f
\big|_\nu^2
\Big\}+
C_\eta\kappa
\mathcal D(\widehat f,\widehat u)
\nonumber\\
&\quad\quad+
\frac{C_\eta\kappa}{1+|\xi|^2}
\bigg\{
\sum_{k=1}^4
|\langle\widehat g,e_k\rangle_v|^2
+
\sum_{1\leq i\leq j\leq3}
\big|
\big\langle
\widehat g,
(v_iv_j-\delta_{ij})\sqrt M
\big\rangle_v
\big|^2
\bigg\},
\label{eq:interaction-estimate}
\end{align}
where we have used
\begin{align*}
|J(\widehat f)|^2\lesssim
\big|
\{\mathbf I-\mathbf P\}\widehat f
\big|_\nu^2,\qquad
|J_\parallel(\widehat f)|^2
+
\big|
\{\mathbf I-\mathbf P_0\}\widehat f
\big|_\nu^2\lesssim
\mathcal D(\widehat f,\widehat u),
\end{align*}
and
\begin{align*}
\big\|
\{\mathbf I-\mathbf P_0\}\widehat f
\big\|_{L_v^2}
\lesssim
\big|
\{\mathbf I-\mathbf P_0\}\widehat f
\big|_\nu.
\end{align*}
To combine the direct dissipation with the compensating term, we rewrite
their sum as
\begin{align}
&\mathcal D(\widehat f,\widehat u)
+
\kappa
\frac{|\xi|^2}{1+|\xi|^2}
\operatorname{Re}
\big\langle
S(\omega)(v\cdot\omega)\widehat f,\widehat f
\big\rangle_v
\nonumber\\
&\quad=
\Big(
1-
\kappa C_0
\frac{|\xi|^2}{1+|\xi|^2}
\Big)
\mathcal D(\widehat f,\widehat u)+
\kappa
\frac{|\xi|^2}{1+|\xi|^2}
\Big\{
\operatorname{Re}
\big\langle
S(\omega)(v\cdot\omega)\widehat f,\widehat f
\big\rangle_v
+
C_0\mathcal D(\widehat f,\widehat u)
\Big\}.
\label{eq:coercive-splitting}
\end{align}
Taking \(\kappa>0\) sufficiently small such that
\begin{align*}
\kappa C_0\leq\frac12,
\end{align*}
we have
\begin{align*}
1-\kappa C_0\frac{|\xi|^2}{1+|\xi|^2}
\geq
1-\kappa C_0
\geq
\frac12.
\end{align*}
Then, it follows from Lemma \ref{lem:modified-coercivity} and
\eqref{eq:coercive-splitting} that
\begin{align}
&\mathcal D(\widehat f,\widehat u)
+
\kappa
\frac{|\xi|^2}{1+|\xi|^2}
\operatorname{Re}
\big\langle
S(\omega)(v\cdot\omega)\widehat f,\widehat f
\big\rangle_v
\nonumber\\
&\qquad\geq
\frac12\mathcal D(\widehat f,\widehat u)
+
c_0\kappa
\frac{|\xi|^2}{1+|\xi|^2}
\Big\{
|a(\widehat f)|^2
+
|\widehat u|^2
+
\big|
\{\mathbf I-\mathbf P\}\widehat f
\big|_\nu^2
\Big\}.
\label{eq:coercive-combination}
\end{align}

Choose \(\eta>0\) sufficiently small such that
\begin{align*}
\eta\leq\frac{c_0}{2},
\end{align*}
and then take \(\kappa>0\), if necessary, such that
\begin{align*}
C_\eta\kappa\leq\frac14.
\end{align*}
Adding \eqref{eq:interaction-estimate} to
\eqref{eq:basic-Fourier} and using
\eqref{eq:coercive-combination}, we obtain that
\begin{align}
&\partial_t
\Big[
\frac12
\big(
\|\widehat f\|_{L_v^2}^2+|\widehat u|^2
\big)
-
\frac{\kappa}{2}
\frac{|\xi|}{1+|\xi|^2}
\big\langle
\mathrm iS(\omega)\widehat f,\widehat f
\big\rangle_v
\Big]
 +
\frac14\mathcal D(\widehat f,\widehat u)
+
\frac{c_0\kappa}{2}
\frac{|\xi|^2}{1+|\xi|^2}
\Big\{
|a(\widehat f)|^2
+
|\widehat u|^2
+
\big|
\{\mathbf I-\mathbf P\}\widehat f
\big|_\nu^2
\Big\}
\nonumber\\
&\quad\leq
\operatorname{Re}
\langle\widehat g,\widehat f\rangle_v
+
\operatorname{Re}
\big(
\widehat G\cdot\overline{\widehat u}
\big)+
\frac{C}{1+|\xi|^2}
\bigg\{
\sum_{k=1}^4
|\langle\widehat g,e_k\rangle_v|^2
+
\sum_{1\leq i\leq j\leq3}
\big|
\big\langle
\widehat g,
(v_iv_j-\delta_{ij})\sqrt M
\big\rangle_v
\big|^2
\bigg\}.
\label{eq:pointwise-Fourier-energy}
\end{align}

Multiplying \eqref{eq:pointwise-Fourier-energy} by
\((1+|\xi|^2)^N\), integrating over \(\xi\), and using Plancherel's theorem,
we obtain from \eqref{intro-relative-dissipation} that
\begin{align}
&\int_{\mathbb R^3}
(1+|\xi|^2)^N
\mathcal D(\widehat f,\widehat u)\,{\rm d}\xi\sim
\|u-J(f)\|_{H^N}^2
+
\sum_{|\alpha|\leq N}
\left\|
\partial_x^\alpha
\{\mathbf I-\mathbf P_0\}f
\right\|_\nu^2,
\label{eq:direct-DN}
\end{align}
where we have used
\begin{align*}
\lambda_1|g|_\nu^2
\leq
-\langle\mathcal Lg,g\rangle_v
\leq
C|g|_\nu^2,
\qquad
g\perp\mathcal W.
\end{align*}
The frequency-weighted part satisfies
\begin{align}
&\int_{\mathbb R^3}
|\xi|^2(1+|\xi|^2)^{N-1}
\Big\{
|a(\widehat f)|^2
+
|\widehat u|^2
+
\big|
\{\mathbf I-\mathbf P\}\widehat f
\big|_\nu^2
\Big\}\,{\rm d}\xi
\nonumber\\
&\qquad\sim
\|\nabla_x\mathbf Pf\|_{L_v^2(H^{N-1})}^2
+
\|\nabla_xu\|_{H^{N-1}}^2+
\sum_{1\leq|\alpha|\leq N}
\left\|
\partial_x^\alpha
\{\mathbf I-\mathbf P\}f
\right\|_\nu^2.
\label{eq:weighted-DN}
\end{align}
Since $\{\mathbf{I} - \mathbf{P}_0\}$ is bounded with respect to the velocity-dissipation norm, namely,  
\begin{align*}
\sum_{1\leq|\alpha|\leq N}
\left\|
\partial_x^\alpha
\{\mathbf I-\mathbf P_0\}f
\right\|_\nu^2
\lesssim
\sum_{1\leq|\alpha|\leq N}
\left\|
\partial_x^\alpha
\{\mathbf I-\mathbf P\}f
\right\|_\nu^2.
\end{align*}
Thus, the sum of \eqref{eq:direct-DN} and \eqref{eq:weighted-DN} is
equivalent to \(\mathcal D_N(t)\) defined in \eqref{TB1}.
By the definitions of \(\mathcal N_{1,N}(t)\) and
\(\mathcal N_{2,N}(t)\) in \eqref{eq:N1} and \eqref{eq:N2}, respectively,
the right-hand side of the integrated inequality is bounded by $
\mathcal N_{1,N}(t)+C\mathcal N_{2,N}(t)$.
Consequently, there exists a constant \(\lambda_3>0\) such that
\begin{align*}
\frac{{\rm d}}{{\rm d}t}\mathcal E_N(t)
+
\lambda_3\mathcal D_N(t)
\leq
\mathcal N_{1,N}(t)
+
C\mathcal N_{2,N}(t),
\end{align*}
which proves \eqref{eq:Fourier-energy-ineq}.

It remains to verify \eqref{eq:energy-equivalence}. Owing to
\eqref{eq:S-bounded}, one has
\begin{align*}
\bigg|
\frac{|\xi|}{1+|\xi|^2}
\big\langle
\mathrm iS(\omega)\widehat f,\widehat f
\big\rangle_v
\bigg|
\lesssim
\|\widehat f\|_{L_v^2}^2.
\end{align*}
Hence, when \(\kappa>0\) is sufficiently small, the interaction term in
\(\mathcal E_N(t)\) is a small perturbation of the standard Fourier
energy. It follows that
\begin{align*}
\mathcal E_N(t)
\sim
\|f(t)\|_{L_v^2(H^N)}^2
+
\|u(t)\|_{H^N}^2,
\end{align*}
which proves \eqref{eq:energy-equivalence} and completes the proof of
Proposition \ref{prop:Fourier-energy}.
\end{proof}

\subsection{Analytic tools}
To estimate the nonlinear terms
\(\mathcal N_{1,N}(t)\) and \(\mathcal N_{2,N}(t)\) arising in
Proposition~\ref{prop:Fourier-energy}, we conclude the section by recording the Sobolev, product, and commutator estimates used below.
\begin{lem}[{\!\!\cite[Lemma 2.1]{CDM-KRM-2011}} and
{\cite[Lemmas 2.1--2.2]{Dk-MZ-1992}}]\label{L2.1}
For any $h_1,h_{2}\in H^3(\mathbb R^3)$ and any multi-index $\alpha$ with
$1\leq|\alpha|\leq3$, it holds that
\begin{align*}
\|h_1\|_{L^{\infty}(\mathbb R^3)} \lesssim&\,
 \|\nabla_{x} h_1\|_{L^2(\mathbb R^3)}^{\frac12}
 \|\nabla^2_{x}h_1\|_{L^2(\mathbb R^3)}^{\frac12},\\
\|h_1h_2\|_{H^1(\mathbb R^3)} \lesssim&\,
 \|h_1\|_{H^2(\mathbb R^3)}\|\nabla_{x} h_2\|_{H^2(\mathbb R^3)},\\
\|\partial^\alpha_{x}(h_1h_2)\|_{L^2(\mathbb R^3)} \lesssim&\,
 \|\nabla_{x} h_1\|_{H^2(\mathbb R^3)}\|\nabla_{x} h_2\|_{H^2(\mathbb R^3)},\\
\|h_1\|_{L^6(\mathbb R^3)} \lesssim&\,
 \|\nabla_{x} h_1\|_{L^2(\mathbb R^3)}\lesssim\|h_1\|_{H^1(\mathbb R^3)},\\
\|h_1\|_{L^q(\mathbb R^3)} \lesssim&\,
 \|h_1\|_{H^1(\mathbb R^3)},\qquad 2\leq q\leq6.
\end{align*}
\end{lem}

\begin{lem}[{{\!\!\cite{commutator1}} and \cite[Appendix]{commutator2}}]\label{L2.2}
Let $h_1$ and $h_2$ be Schwartz functions. For $k\geq 1$, one has
\begin{align*}
\|\nabla^{k}(h_1h_2) \|_{L^r(\mathbb{R}^3)} \lesssim&\, \|h_1\|_{L^{r_1}(\mathbb{R}^3) }\|\nabla^{k}h_2\|_{L^{r_2} }+ \|h_2\|_{L^{r_3}(\mathbb{R}^3) }\|\nabla^{k}h_1\|_{L^{r_4} (\mathbb{R}^3)},\\
\|\nabla^{k}(h_1h_2)-h_1\nabla^k h_2 \|_{L^r(\mathbb{R}^3)} \lesssim&\, \|\nabla h_1\|_{L^{r_1}(\mathbb{R}^3)}\|\nabla^{k-1}h_2\|_{L^{r_2}(\mathbb{R}^3)}+ \|h_2\|_{L^{r_3}(\mathbb{R}^3)}\|\nabla^{k}h_1\|_{L^{r_4}(\mathbb{R}^3)},    
\end{align*}
where $1<r,r_2,r_4<\infty$ and $r_i(1\leq i\leq 4)$ satisfy 
\begin{align*}
\frac{1}{r_1}+\frac{1}{r_2}=\frac{1}{r_3}+\frac{1}{r_4}=\frac{1}{r}.   
\end{align*}
\end{lem}

\section{Global classical solutions}

In this section, we prove global existence for the Cauchy problem \eqref{A1}--\eqref{A-1}. The main task is to control the nonlinear terms \(\mathcal{N}_{1,N}(t)\) and \(\mathcal{N}_{2,N}(t)\) arising in Proposition \ref{prop:Fourier-energy}, which allows us to close the compensated energy inequality \eqref{eq:Fourier-energy-ineq} and extend the local solution globally in time.

\subsection{Estimates of the nonlinear terms}

We first estimate \(\mathcal N_{1,N}(t)\). The main point is to combine the kinetic nonlinearity with the nonlinear drag term before applying the product estimates.

\begin{lem}\label{lem:N1}
Let \(N\geq4\). For the  strong solution $(u,f)$  to the Cauchy problem \eqref{A1}--\eqref{A-1}, we
have
\begin{align}\label{eq:N1-bound}
|\mathcal N_{1,N}(t)|
\leq
C\sqrt{\mathcal E_N(t)}\,\mathcal D_N(t).
\end{align}
\end{lem}

\begin{proof}
There exist positive constants \(c_{N,\alpha}\) such that
\begin{align*}
(1+|\xi|^2)^N
=
\sum_{|\alpha|\leq N}
c_{N,\alpha}|\xi^\alpha|^2.
\end{align*}
Hence, by Plancherel's theorem and the self-adjointness of the Leray
projection, we have
\begin{align*}
\mathcal N_{1,N}(t)
=
\sum_{|\alpha|\leq N}c_{N,\alpha}
\Big\{&
\operatorname{Re}
\big(
\partial_x^\alpha g,\partial_x^\alpha f
\big)_{x,v}
-
\operatorname{Re}
\big(
\partial_x^\alpha(a(f)u),\partial_x^\alpha u
\big)_x-
\operatorname{Re}
\big(
\partial_x^\alpha(u\cdot\nabla_xu),
\partial_x^\alpha u
\big)_x
\Big\}.
\end{align*}

We first combine the kinetic term with the nonlinear drag term. Recall
that $
g
=
-u\cdot\nabla_vf
+
\frac12(u\cdot v)f$ in \eqref{eq:gG}.
In view of the Leibniz rule, one gets
\begin{align*}
\partial_x^\alpha g
=
\sum_{\beta\leq\alpha}
\binom{\alpha}{\beta}
\Big\{
-\partial_x^\beta u\cdot\nabla_v
\partial_x^{\alpha-\beta}f
+
\frac12
\big(
\partial_x^\beta u\cdot v
\big)
\partial_x^{\alpha-\beta}f
\Big\}.
\end{align*}
Combining the kinetic and fluid contributions yields
\begin{align*} 
\mathcal N_{1,N}(t)=I_1+I_2+I_3,
\end{align*} 
where
\begin{align*}
I_1
:=&\,
\sum_{|\alpha|\leq N}c_{N,\alpha}
\sum_{\beta\leq\alpha}
\binom{\alpha}{\beta}
\operatorname{Re}
\int_{\mathbb R^3}
a\big(
\partial_x^{\alpha-\beta}f
\big)
\partial_x^\beta u
\cdot
\overline{
J(\partial_x^\alpha f)-\partial_x^\alpha u
}
\,{\rm d}x,\\
I_2
:=&\,
\sum_{|\alpha|\leq N}c_{N,\alpha}
\sum_{\beta\leq\alpha}
\binom{\alpha}{\beta}
\operatorname{Re}
\big(
-\partial_x^\beta u\cdot\nabla_v
\{\mathbf I-\mathbf P\}
\partial_x^{\alpha-\beta}f +
\frac12
\big(
\partial_x^\beta u\cdot v
\big)
\{\mathbf I-\mathbf P\}
\partial_x^{\alpha-\beta}f,
\{\mathbf I-\mathbf P_0\}
\partial_x^\alpha f
\big)_{x,v},\\
I_3
:=&\,
-
\sum_{|\alpha|\leq N}c_{N,\alpha}
\operatorname{Re}
\big(
\partial_x^\alpha(u\cdot\nabla_xu),
\partial_x^\alpha u
\big)_x.
\end{align*}

Next, we estimate \(I_1\), \(I_2\), and \(I_3\) in turn. By
Lemma~\ref{L2.1}, the equivalence of the compensated energy, and the
definition of \(\mathcal D_N(t)\), one has
\begin{align}
&\|a(f)\|_{H^N}
+\|J(f)\|_{H^N}
+\|u\|_{H^N}
+\|\nabla_xu\|_{L^\infty}
\lesssim
\sqrt{\mathcal E_N(t)},
\label{eq:basic-energy-control}\\
&\|a(f)\|_{L^\infty}
+\|J(f)\|_{L^\infty}
+\|u\|_{L^\infty}
+\|u-J(f)\|_{H^N}+
\bigg(
\sum_{|\alpha|\leq N}
\left\|
\partial_x^\alpha
\{\mathbf I-\mathbf P_0\}f
\right\|_\nu^2
\bigg)^{\frac{1}{2}}
\lesssim
\sqrt{\mathcal D_N(t)} .
\label{eq:basic-dissipation-control}
\end{align}
For the term \(I_1\),  by using Lemma \ref{L2.2} and the estimates
\eqref{eq:basic-energy-control}--\eqref{eq:basic-dissipation-control}, we have
\begin{align}
|I_1|
\lesssim&\,
\|a(f)u\|_{H^N}
\|u-J(f)\|_{H^N}
\nonumber\\
\lesssim&\,
\big(
\|a(f)\|_{H^N}\|u\|_{L^\infty}
+
\|a(f)\|_{L^\infty}\|u\|_{H^N}
\big)
\|u-J(f)\|_{H^N}
\nonumber\\
\lesssim&\,
\sqrt{\mathcal E_N(t)}\,\mathcal D_N(t).
\label{eq:I1-bound}
\end{align}
For the term \(I_2\), by leveraging $\{\mathbf I-\mathbf P\}f
=
J(f)\cdot v\sqrt M
+
\{\mathbf I-\mathbf P_0\}f$
and applying Lemmas \ref{L2.1}--\ref{L2.2}, we obtain that
\begin{align*}
&\|uJ(f)\|_{H^N}
+
\bigg(
\sum_{|\alpha|\leq N}
\left\|
\partial_x^\alpha
\big(
u\{\mathbf I-\mathbf P_0\}f
\big)
\right\|_\nu^2
\bigg)^{\frac{1}{2}}
\lesssim
\sqrt{\mathcal E_N (t)\, {\mathcal D_N}(t)},
\end{align*}
which gives rise to
\begin{align}
|I_2|
\lesssim&\,
\bigg[
\|uJ(f)\|_{H^N}
+
\bigg(
\sum_{|\alpha|\leq N}
\left\|
\partial_x^\alpha
\big(
u\{\mathbf I-\mathbf P_0\}f
\big)
\right\|_\nu^2
\bigg)^{\frac{1}{2}}
\bigg]\times
\bigg(
\sum_{|\alpha|\leq N}
\left\|
\partial_x^\alpha
\{\mathbf I-\mathbf P_0\}f
\right\|_\nu^2
\bigg)^{\frac{1}{2}}\nonumber\\
\lesssim&\,
\sqrt{\mathcal E_N(t)}\, \mathcal D_N(t).
\label{eq:I2-bound}
\end{align}

Finally, we deal with the remaining term $I_3$. Since the divergence-free condition $\nabla_x\cdot u = 0$ holds,
the zero-order contribution to $I_3$ vanishes. For $1\leq|\alpha|\leq N$, we have
\begin{align*}
\big(\partial_x^\alpha(u\cdot\nabla_xu),\partial_x^\alpha u\big)_x=\big([\partial_x^\alpha,u\cdot\nabla_x]u,\partial_x^\alpha u\big)_x.
\end{align*}
Therefore, from Lemma \ref{L2.2}, one gets
\begin{align}
|I_3|
 \lesssim
\|\nabla_xu\|_{L^\infty}
\|\nabla_xu\|_{H^{N-1}}^2
\lesssim
\sqrt{\mathcal E_N(t)}\, \mathcal D_N(t).
\label{eq:I3-bound}
\end{align}
Combining all the estimates \eqref{eq:I1-bound}--\eqref{eq:I3-bound}, we conclude that
\begin{align*}
|\mathcal N_{1,N}(t)|
\leq
C\sqrt{\mathcal E_N(t)}\, \mathcal D_N(t),
\end{align*}
which leads to \eqref{eq:N1-bound}.
\end{proof}

We next estimate the velocity moments entering
\(\mathcal N_{2,N}(t)\).

\begin{lem}\label{lem:N2}
Let \(N\geq4\). For the strong solution \((u,f)\) to the Cauchy problem
\eqref{A1}--\eqref{A-1}, we have
\begin{align}\label{eq:N2-bound}
\mathcal N_{2,N}(t)
\leq
C\mathcal E_N(t) \mathcal D_N(t).
\end{align}
\end{lem}

\begin{proof}
For any Schwartz function $\psi = \psi(v)$, integration by parts in $v$ yields  
\begin{align}\label{eq:moment-source}
\langle g,\psi\rangle_v
=\langle -u\cdot\nabla_vf
+
\frac12(u\cdot v)f, \psi\rangle_{v}=
u\cdot
\big\langle
f,\nabla_v\psi+\frac12v\psi
\big\rangle_v.
\end{align}
Applying \eqref{eq:moment-source} to the velocity profiles appearing in
\(\mathcal N_{2,N}(t)\), we obtain that
\begin{align*}
\langle g,\sqrt M\rangle_v
&=0,
\\
\langle g,v_k\sqrt M\rangle_v
&=
u_k a(f),
\qquad \qquad\qquad \quad 1\leq k\leq3,
\\
\big\langle
g,(v_iv_j-\delta_{ij})\sqrt M
\big\rangle_v
&=
u_iJ_j(f)+u_jJ_i(f),
\qquad \,1\leq i\leq j\leq3.
\end{align*}
Recalling that \(e_1=\sqrt M\) and
\(e_{k+1}=v_k\sqrt M\), the definition of
\(\mathcal N_{2,N}(t)\) and Plancherel's theorem yield
\begin{align}
\mathcal N_{2,N}(t)
&\lesssim
\|a(f)u\|_{H^{N-1}}^2
+
\sum_{1\leq i\leq j\leq3}
\|u_iJ_j(f)+u_jJ_i(f)\|_{H^{N-1}}^2
\nonumber\\
&\lesssim
\|a(f)u\|_{H^{N-1}}^2
+
\|uJ(f)\|_{H^{N-1}}^2.
\label{eq:N2-product}
\end{align}
By using Lemmas \ref{L2.1}--\ref{L2.2}, one gets
\begin{align}\label{G3.12}
\|a(f)u\|_{H^{N-1}}
+
\|uJ(f)\|_{H^{N-1}}
\lesssim&\,
\|u\|_{H^{N-1}}
\big(
\|a(f)\|_{L^\infty}
+
\|J(f)\|_{L^\infty}
\big)+
\|u\|_{L^\infty}
\big(
\|a(f)\|_{H^{N-1}}
+
\|J(f)\|_{H^{N-1}}
\big)\nonumber\\
\lesssim&\,
\sqrt{\mathcal E_N(t)}\,\sqrt{\mathcal D_N (t)}.
\end{align}
Substituting the estimate \eqref{G3.12} into \eqref{eq:N2-product} gives
\begin{align*}
\mathcal N_{2,N}(t)
\lesssim
\mathcal E_N(t)\mathcal D_N(t),
\end{align*}
which results in \eqref{eq:N2-bound}.
\end{proof}

With Lemmas \ref{lem:N1} and \ref{lem:N2} in hand, we are now ready to close the nonlinear energy estimate.

\begin{thm}\label{thm:closed-energy}
Let \(N\geq4\), and let \((u,f)\) be a strong solution to the Cauchy
problem \eqref{A1}--\eqref{A-1}. If
\begin{align*}
\mathcal E_N(t)\leq1,
\end{align*}
then there exists a constant \(\lambda_4>0\), such that
\begin{align}\label{eq:closed-energy}
\frac{{\rm d}}{{\rm d}t}\mathcal E_N(t)
+
\lambda_4\mathcal D_N(t)
\leq
C\sqrt{\mathcal E_N(t)}\,\mathcal D_N(t).
\end{align}
\end{thm}

\begin{proof}
It follows from
Proposition \ref{prop:Fourier-energy} and
Lemmas \ref{lem:N1}--\ref{lem:N2}  that
\begin{align*}
\frac{{\rm d}}{{\rm d}t}\mathcal E_N(t)
+
\lambda_2\mathcal D_N(t)
\leq
C\big(
\sqrt{\mathcal E_N(t)}
+
\mathcal E_N(t)
\big) \mathcal D_N(t).
\end{align*}
Since \(\mathcal E_N(t)\leq1\), one has
\begin{align*}
\mathcal E_N(t)
\leq
\sqrt{\mathcal E_N(t)}.
\end{align*}
After enlarging \(C\) and choosing the condition \(0 < \lambda_4 \leq \lambda_2\), we obtain \eqref{eq:closed-energy}.
\end{proof}

\subsection{Proof of global existence}

We now complete the proof of global existence for the Euler-VFP Cauchy problem \eqref{A1}--\eqref{A-1}. Local existence and uniqueness of
classical solutions follow from a standard regularization and iteration
scheme for fluid-particle systems; see \cite{CDM-KRM-2011} for a closely
related construction. Moreover, if
\(F_0(x,v)=M+\sqrt{M}f_0(x,v)\geq0\), the maximum principle for the kinetic
equation, as in \cite{Gy-IUMJ-2004}, preserves the non-negativity of the particle distribution. Consequently, we obtain
\begin{align*}
F(t,x,v)=M+\sqrt{M}f(t,x,v)\geq0.    
\end{align*}

Let
\(T_{\max}>0\) denote the maximal existence time of the local solution.
Choose \(\varepsilon_0>0\) sufficiently small such that
\begin{align}\label{eq:epsilon-choice}
2\varepsilon_0\leq1,
\qquad
\sqrt{2\varepsilon_0}
\leq
\frac{\lambda_4}{2C},
\end{align}
and assume that
\begin{align*}
\mathcal E_N(0)\leq\varepsilon_0.
\end{align*}
Define
\begin{align*}
T^*
:=
\sup\big\{
T<T_{\max}:
\sup_{0\leq t\leq T}\mathcal E_N(t)
\leq2\varepsilon_0
\big\}.
\end{align*}
For \(0\leq t< T^*\), one has
\(\mathcal E_N(t)\leq2\varepsilon_0\leq1\). Therefore, from
Theorem  \ref{thm:closed-energy} and \eqref{eq:epsilon-choice}, we further obtain 
\begin{align*}
\frac{{\rm d}}{{\rm d}t}\mathcal E_N(t)
+
\big(
\lambda_4-C\sqrt{\mathcal E_N(t)}
\big) \mathcal D_N(t)
\leq0,
\end{align*}
which yields
\begin{align}\label{G3.15}
\frac{{\rm d}}{{\rm d}t}\mathcal E_N(t)
+
\frac{\lambda_4}{2} \mathcal D_N(t)
\leq0.
\end{align}
Integrating \eqref{G3.15} over \([0,t]\), we arrive at
\begin{align}\label{eq:global-energy-bound}
\mathcal E_N(t)
+
\frac{\lambda_4}{2}
\int_0^t\mathcal D_N(\tau)\,{\rm d}\tau
\leq
\mathcal E_N(0)
\leq
\varepsilon_0,
\end{align}
for all $0\leq t< T^*$.
In particular,
\begin{align*}
\sup_{0\leq t\leq T^*}\mathcal E_N(t)
\leq
\varepsilon_0
<
2\varepsilon_0,
\end{align*}
which strictly improves the bootstrap assumption.  By continuity of $\mathcal{E}_{N}(t)$, it follows that \(T^* = T_{\max}\).

It remains to exclude the possibility that \(T_{\max}<\infty\).
By \eqref{eq:global-energy-bound} and the equivalence of
\(\mathcal E_N(t)\) with the standard \(H^N\) energy, we have
\begin{align*}
\sup_{0\leq t<T_{\max}}\mathcal E_N(t)
\leq
\mathcal E_N(0)
\leq
\varepsilon_0.
\end{align*}
Hence, the classical solution \((u,f)\) remains uniformly bounded in
the solution space of the local existence theory throughout
\([0,T_{\max})\).
The continuation criterion therefore allows \((u,f)\) to be extended
beyond \(T_{\max}\), contradicting its maximality. Thus,
\(T_{\max}=\infty\).
Consequently, \((u,f)\) is a global classical solution to
\eqref{A1}--\eqref{A-1} and satisfies
\begin{align*}
\mathcal E_N(t)
+
\frac{\lambda_4}{2}
\int_0^t\mathcal D_N(\tau)\,{\rm d}\tau
\leq
\mathcal E_N(0),
\end{align*}
for all $t\geq 0$,
which proves the global existence and the uniform energy estimate
\eqref{TA2} in Theorem~\ref{Th1}. \hfill\(\Box\)

\section{Time decay estimates}

In this section, we prove the decay estimate \eqref{TA3} in Theorem \ref{Th1}.Following the positive-order instant-energy method developed in   \cite[Section 4]{YangYu2010} and the related decay arguments in \cite{UYZ-CAM-2006,YZ-CMP-2006}, we consider an energy functional involving at least one spatial derivative. This choice is essential: without any additional low-frequency assumption on the initial data, the full zero-order \(L^2\) energy is not known to be integrable in time, whereas the positive-order energy is controlled by the global dissipation. The Gr\"{o}nwall-type argument in \cite{Dk-MZ-1992} can then be used to derive its algebraic decay.

For the same sufficiently small constant \(\kappa > 0\) as in Proposition~\ref{prop:Fourier-energy}, we define positive-order energy  \(\mathcal{E}_{N}^{(1)}(t)\) by  
\begin{align}\label{eq:positive-energy}
\mathcal E_N^{(1)}(t)
:=
\int_{\mathbb R^3}
\Big\{
\frac12
|\xi|^2(1+|\xi|^2)^{N-2}
\big(
\|\widehat f\|_{L_v^2}^2
+
|\widehat u|^2
\big)
 -
\frac{\kappa}{2}
|\xi|^3(1+|\xi|^2)^{N-3}
\big\langle
\mathrm iS(\omega)\widehat f,\widehat f
\big\rangle_v
\Big\}\,{\rm d}\xi.
\end{align}
Since
\begin{align*}
\frac{
|\xi|^3(1+|\xi|^2)^{N-3}
}{
|\xi|^2(1+|\xi|^2)^{N-2}
}
=
\frac{|\xi|}{1+|\xi|^2}
\leq
\frac12,
\end{align*}
the boundedness of \(S(\omega)\) gives
\begin{align}\label{eq:positive-equivalence}
\mathcal E_N^{(1)}(t)
\sim
\sum_{1\leq|\alpha|\leq N-1}
\left\{
\|\partial_x^\alpha f(t)\|_{L_{x,v}^2}^2
+
\|\partial_x^\alpha u(t)\|_{L^2}^2
\right\}.
\end{align}
Correspondingly, we define the dissipation rate $\mathcal{D}^{(1)}_{N}(t)$ as
\begin{align*} 
\mathcal D_N^{(1)}(t)
:=&\,
\sum_{1\leq|\alpha|\leq N-1}
\left\{
\left\|
\partial_x^\alpha\{u-J(f)\}
\right\|_{L^2}^2
+
\left\|
\partial_x^\alpha
\{\mathbf I-\mathbf P_0\}f
\right\|_\nu^2
\right\}
\nonumber\\
&+
\sum_{2\leq|\alpha|\leq N-1}
\left\{
\|\partial_x^\alpha\mathbf Pf\|_{L_{x,v}^2}^2
+
\left\|
\partial_x^\alpha
\{\mathbf I-\mathbf P\}f
\right\|_\nu^2
+
\|\partial_x^\alpha u\|_{L^2}^2
\right\}.
\end{align*}

The following positive-order energy estimate is the key step in the decay analysis.

\begin{thm}\label{thm:positive-energy}
Let \(N\geq4\), and let \((u,f)\) be the global classical solution to the
Cauchy problem \eqref{A1}--\eqref{A-1}. Then there exist constants
\(\lambda_5>0\) and \(C>0\) such that
\begin{align}\label{eq:positive-differential}
\frac{{\rm d}}{{\rm d}t}\mathcal E_N^{(1)}(t)
+
\lambda_5\mathcal D_N^{(1)}(t)
\leq
C\mathcal D_N(t)\mathcal E_N^{(1)}(t)
+
C\bigl(\mathcal E_N^{(1)}(t)\bigr)^3.
\end{align}
\end{thm}

\begin{proof}
Multiplying  \eqref{eq:basic-Fourier} by $
|\xi|^2(1+|\xi|^2)^{N-2}$,
multiplying   \eqref{eq:interaction-identity} by
$
\kappa|\xi|^2(1+|\xi|^2)^{N-3}
$,
and then adding the resulting identities, we obtain
\begin{align}
&\frac{{\rm d}}{{\rm d}t}\mathcal E_N^{(1)}(t)
+
\lambda_5\mathcal D_N^{(1)}(t)\leq
\mathcal R_{1,N}^{(1)}(t)
+
C\mathcal R_{2,N}^{(1)}(t)
\label{eq:positive-preliminary}
\end{align}
for some constant \(\lambda_5>0\), where
\begin{align*}
\mathcal R_{1,N}^{(1)}(t)
:=&\,
\int_{\mathbb R^3}
|\xi|^2(1+|\xi|^2)^{N-2}
\big\{
\operatorname{Re}
\langle\widehat g,\widehat f\rangle_v
+
\operatorname{Re}
\big(
\widehat G\cdot\overline{\widehat u}
\big)
\big\}\,{\rm d}\xi,\\
\mathcal R_{2,N}^{(1)}(t)
:=&\,
\sum_{k=1}^4
\int_{\mathbb R^3}
|\xi|^2(1+|\xi|^2)^{N-3}
|\langle\widehat g,e_k\rangle_v|^2\,{\rm d}\xi
\nonumber\\
&+
\sum_{1\leq i\leq j\leq3}
\int_{\mathbb R^3}
|\xi|^2(1+|\xi|^2)^{N-3}
\big|
 \big\langle
\widehat g,
(v_iv_j-\delta_{ij})\sqrt M
 \big\rangle_v
\big|^2
\,{\rm d}\xi.
\end{align*}
Here, the linear terms have been treated exactly as in the proof of Proposition \ref{prop:Fourier-energy}. More precisely,
\begin{align*}
\sum_{1\leq|\alpha|\leq N-1}
\left\{
\left\|
\partial_x^\alpha\{u-J(f)\}
\right\|_{L^2}^2
+
\left\|
\partial_x^\alpha
\{\mathbf I-\mathbf P_0\}f
\right\|_\nu^2
\right\}\lesssim\int_{\mathbb R^3}
|\xi|^2(1+|\xi|^2)^{N-2}
\mathcal D(\widehat f,\widehat u)\,{\rm d}\xi,
\end{align*}
and
\begin{align*}
&\int_{\mathbb R^3}
|\xi|^4(1+|\xi|^2)^{N-3}
\Big\{
|a(\widehat f)|^2
+
|\widehat u|^2
+
\big|
\{\mathbf I-\mathbf P\}\widehat f
\big|_\nu^2
\Big\}\,{\rm d}\xi
\\
&\quad\sim
\sum_{2\leq|\alpha|\leq N-1}
\left\{
\|\partial_x^\alpha\mathbf Pf\|_{L_{x,v}^2}^2
+
\|\partial_x^\alpha u\|_{L^2}^2
+
\left\|
\partial_x^\alpha
\{\mathbf I-\mathbf P\}f
\right\|_\nu^2
\right\}.
\end{align*}
The two estimates mentioned above provide the dissipation
\(\mathcal D_N^{(1)}(t)\) in \eqref{eq:positive-preliminary}.

We first estimate \(\mathcal R_{1,N}^{(1)}(t)\). There exist positive constants \(c_{N,\alpha}^{(1)}\) such that
\begin{align*}
|\xi|^2(1+|\xi|^2)^{N-2}
=
\sum_{1\leq|\alpha|\leq N-1}
c_{N,\alpha}^{(1)}|\xi^\alpha|^2.
\end{align*}
Since $
G
=
-\mathbb P(u\cdot\nabla_xu)
-
\mathbb P(a(f)u)$ in \eqref{eq:gG}, with the help of the self-adjointness of \(\mathbb P\), the commutation of \(\mathbb P\) with spatial derivatives, and the fact that \(\mathbb Pu = u\), we derive
\begin{align}\label{eq:positive-R1-physical}
\mathcal R_{1,N}^{(1)}
=
\sum_{1\leq|\alpha|\leq N-1}
c_{N,\alpha}^{(1)}
\Big\{&
\operatorname{Re}
\big(
\partial_x^\alpha g,\partial_x^\alpha f
\big)_{x,v}-
\operatorname{Re}
\big(
\partial_x^\alpha(a(f)u),\partial_x^\alpha u
\big)_x-
\operatorname{Re}
\big(
\partial_x^\alpha(u\cdot\nabla_xu),
\partial_x^\alpha u
\big)_x
\Big\}.
\end{align}
Repeating the decomposition used in Lemma \ref{lem:N1} gives
\begin{align*}
&\operatorname{Re}
\big(
\partial_x^\alpha g,\partial_x^\alpha f
\big)_{x,v}
-
\operatorname{Re}
\big(
\partial_x^\alpha(a(f)u),\partial_x^\alpha u
\big)_x
\nonumber\\
=&\,
\sum_{\beta\leq\alpha}
\binom{\alpha}{\beta}
\operatorname{Re}
\int_{\mathbb R^3}
a(\partial_x^{\alpha-\beta}f)
\partial_x^\beta u
\cdot
\overline{
J(\partial_x^\alpha f)-\partial_x^\alpha u
}
\,{\rm d}x
\nonumber\\
&+
\sum_{\beta\leq\alpha}
\binom{\alpha}{\beta}
\operatorname{Re}
\Big(
-\partial_x^\beta u\cdot\nabla_v
\{\mathbf I-\mathbf P\}
\partial_x^{\alpha-\beta}f
+
\frac12
\big(
\partial_x^\beta u\cdot v
\big)
\{\mathbf I-\mathbf P\}
\partial_x^{\alpha-\beta}f,
\{\mathbf I-\mathbf P_0\}
\partial_x^\alpha f
\Big)_{x,v}\nonumber\\
=:&\,J_1+J_2.
\end{align*}
For the first term $J_1$, by Lemmas \ref{L2.1}--\ref{L2.2}, one has
\begin{align}\label{G4.8}
\sum_{1\leq|\alpha|\leq N-1}
 c_{N,\alpha}^{(1)}|J_1|
\leq&\,C
\|\nabla_x(a(f)u)\|_{H^{N-2}}
\|\nabla_x(u-J(f))\|_{H^{N-2}}\nonumber\\
\leq&\, C\sqrt{\mathcal D_N (t)\,
 \mathcal E_N^{(1)} (t)\,  \mathcal D_N^{(1)}(t)}\nonumber\\
 \leq&\, \frac{\lambda_5}{16}\mathcal D_N^{(1)}(t)
+
C\mathcal D_N(t)\,\mathcal E_N^{(1)}(t).
\end{align}
For the second term \(J_2\), the same product estimates yield
\begin{align*} 
\sum_{1\leq|\alpha|\leq N-1}
c_{N,\alpha}^{(1)}|J_2|
\lesssim&\,
\bigg[
\|\nabla_x(uJ(f))\|_{H^{N-2}}+
\|u\|_{L^\infty}
\bigg(
\sum_{1\leq|\alpha|\leq N-1}
\left\|
\partial_x^\alpha
\{\mathbf I-\mathbf P_0\}f
\right\|_\nu^2
\bigg)^{\frac{1}{2}}
\nonumber\\
&\quad+
\|\nabla_xu\|_{H^{N-2}}
\left\|
|\{\mathbf I-\mathbf P_0\}f|_{\nu}
\right\|_{L^\infty}
\bigg]\times
\bigg(
\sum_{1\leq|\alpha|\leq N-1}
\left\|
\partial_x^\alpha
\{\mathbf I-\mathbf P_0\}f
\right\|_\nu^2
\bigg)^{\frac{1}{2}}\nonumber\\
\lesssim&\,
\sqrt{\mathcal E_N(t)}\,
\mathcal D_N^{(1)}(t)+\sqrt{\mathcal D_N(t)\,
 \mathcal E_N^{(1)}(t)\,  
 \mathcal D_N^{(1)}(t)}.
\end{align*}
Since \(\mathcal E_N(t)\) is uniformly small, we   make use of Young's inequality to obtain
\begin{align}
\sum_{1\leq|\alpha|\leq N-1}
c_{N,\alpha}^{(1)}|J_2|
\leq
\frac{\lambda_{5}}{16}\mathcal D_N^{(1)}(t)
+
C \mathcal D_N(t)\,\mathcal E_N^{(1)}(t).
\label{eq:positive-J2}
\end{align}

We next consider the Euler transport part in \eqref{eq:positive-R1-physical}. Since
\(\nabla_x\cdot u=0\), it follows that
\begin{align*}
\big(
u\cdot\nabla_x\partial_x^\alpha u,
\partial_x^\alpha u
\big)_x
=0.
\end{align*}
For \(2\leq|\alpha|\leq N-1\), from Lemma \ref{L2.2}, we have
\begin{align}\label{eq:positive-higher-Euler}
\sum_{2\leq|\alpha|\leq N-1}
c_{N,\alpha}^{(1)}
\left|
\big(
\partial_x^\alpha(u\cdot\nabla_xu),
\partial_x^\alpha u
\big)_x
\right|\lesssim&\,
\|\nabla_xu\|_{L^\infty}
\|\nabla_xu\|_{H^{N-2}}
\|\nabla_x^2u\|_{H^{N-3}}\nonumber\\
\lesssim&\, \frac{\lambda_{5}}{16}\mathcal D_N^{(1)}(t)+C\mathcal{D}_{N}(t)\mathcal{E}_{N}^{(1)}(t).
\end{align}
For \(|\alpha|=1\), integration by parts gives
\begin{align}
\left|
\big(
\nabla_x(u\cdot\nabla_xu),
\nabla_xu
\big)_x
\right|
&\leq
\frac{\lambda_{5}}{16} \|\nabla_x^2u\|_{L^2}^2
+
C \|\nabla_xu\|_{L^2}^6
\nonumber\\
&\leq
\frac{\lambda_{5}}{16} \mathcal D_N^{(1)}(t)
+
C 
\big(\mathcal E_N^{(1)}(t)\big)^3.
\label{eq:positive-first-Euler}
\end{align}
Combining the estimates \eqref{eq:positive-higher-Euler} and 
\eqref{eq:positive-first-Euler} yields
\begin{align}\label{G4.13}
\sum_{1\leq|\alpha|\leq N-1}
c_{N,\alpha}^{(1)}
\left|
\big(
\partial_x^\alpha(u\cdot\nabla_xu),
\partial_x^\alpha u
\big)_x
\right|
\lesssim&\, \frac{\lambda_{5}}{8}\mathcal D_N^{(1)}(t)+C\mathcal{D}_{N}(t)\mathcal{E}_{N}^{(1)}(t)+C 
\big(\mathcal E_N^{(1)}(t)\big)^3.
\end{align}

It remains to estimate \(\mathcal R_{2,N}^{(1)}(t)\). By using a similar argument in Lemma \ref{lem:N2} and Plancherel's theorem, we can infer that
\begin{align}
\mathcal R_{2,N}^{(1)}
\lesssim
\|\nabla_x(a(f)u)\|_{H^{N-3}}^2
+
\|\nabla_x(uJ(f))\|_{H^{N-3}}^2\lesssim
\mathcal D_N(t)\mathcal E_N^{(1)}(t).
\label{eq:positive-moment-final}
\end{align}
Combining \eqref{G4.8},
\eqref{eq:positive-J2}, \eqref{G4.13}, and
\eqref{eq:positive-moment-final}, we end up with
\begin{align*}
\left|
\mathcal R_{1,N}^{(1)}
\right|
+
C\mathcal R_{2,N}^{(1)}
\leq 
\frac{\lambda_{5}}{4}\mathcal D_N^{(1)}(t)
+
C \mathcal D_N(t)\mathcal E_N^{(1)}(t)
+
C 
\big (\mathcal E_N^{(1)}(t)\big )^3,
\end{align*}
which together with \eqref{eq:positive-preliminary} leads to
\eqref{eq:positive-differential}. Therefore, the proof of Theorem \ref{thm:positive-energy} is completed.
\end{proof}

In order to obtain the decay estimates \eqref{TA3}, we shall utilize the following Gr\"{o}nwall-type lemma:

\begin{lem}[{\!\!\cite{Dk-MZ-1992}; see also
\cite[Lemma~4.2]{YangYu2010}}]\label{lem:decay-Gronwall}
Let \(t_0>0\) and let \(y\in C^1([t_0,\infty))\) be a nonnegative
function satisfying
\begin{align*}
y'(t)\leq a(t)y(t),
\qquad t\geq t_0,
\end{align*}
where \(a(t)\geq0\). Assume that
\begin{align*}
A:=\int_{t_0}^{\infty}y(t)\,{\rm d}t<\infty,
\qquad
B:=\int_{t_0}^{\infty}a(t)\,{\rm d}t<\infty.
\end{align*}
Then
\begin{align}\label{eq:decay-Gronwall}
y(t)
\leq
\frac{
\bigl(t_0y(t_0)+1\bigr)\exp(A+B)-1
}{t},
\qquad
t\geq t_0.
\end{align}
\end{lem}

We now prove the pointwise decay estimate \eqref{TA3}, thereby completing the proof of Theorem \ref{Th1}.

\begin{proof}[Proof of \eqref{TA3} in Theorem \ref{Th1}]
By employing \eqref{eq:positive-equivalence} and the definition of
\(\mathcal D_N(t)\) in \eqref{TB1},  we have
\begin{align}
\mathcal E_N^{(1)}(t)
\lesssim
\|\nabla_x\mathbf Pf\|_{L_v^2(H^{N-2})}^2
+
\|\nabla_xu\|_{H^{N-2}}^2+
\sum_{1\leq|\alpha|\leq N-1}
\left\|
\partial_x^\alpha
\{\mathbf I-\mathbf P\}f
\right\|_\nu^2
\lesssim
\mathcal D_N(t).
\label{eq:Eplus-by-D}
\end{align}
It follows from the global energy estimate
\eqref{eq:global-energy-bound} that
\begin{align}
\int_0^\infty
\mathcal E_N^{(1)}(\tau)\,{\rm d}\tau
\lesssim
\int_0^\infty
\mathcal D_N(\tau)\,{\rm d}\tau
\lesssim
\mathcal E_N(0).
\label{eq:positive-integrable}
\end{align}
Moreover, based on \eqref{eq:positive-equivalence} and \eqref{eq:energy-equivalence}, we also have
\begin{align}
\sup_{t\geq0}\mathcal E_N^{(1)}(t)
\lesssim
\sup_{t\geq0}\mathcal E_N(t)
\lesssim
\mathcal E_N(0).
\label{eq:positive-uniform}
\end{align}
Let's define 
\begin{align*}
y(t):=\mathcal E_N^{(1)}(t).
\end{align*}
Dropping the nonnegative dissipation term from
\eqref{eq:positive-differential}, we obtain that
\begin{align*}
y'(t)
\leq
a_*(t)y(t),
\quad\text{where}\quad
a_*(t)
:=
C\big(
\mathcal D_N(t)+y(t)^2
\big).
\end{align*}
which, together with \eqref{eq:positive-integrable} and \eqref{eq:positive-uniform}, gives  
\begin{align*}
\int_0^\infty a_*(t)\,{\rm d}t
&\lesssim
\int_0^\infty\mathcal D_N(t)\,{\rm d}t
+
\int_0^\infty y(t)^2\,{\rm d}t
\nonumber\\
&\lesssim
\mathcal E_N(0)
+
\left(
\sup_{t\geq0}y(t)
\right)
\int_0^\infty y(t)\,{\rm d}t
\nonumber\\
&\lesssim
\mathcal E_N(0).
\end{align*}

Next,
we apply Lemma \ref{lem:decay-Gronwall} with \(t_0=1\). Denote
\begin{align*}
A_1
:=
\int_1^\infty y(t)\,{\rm d}t,
\qquad
B_1
:=
\int_1^\infty a_*(t)\,{\rm d}t.
\end{align*}
Then, for \(t\geq1\), it holds that
\begin{align*}
y(t)
\leq
\frac{
\bigl(y(1)+1\bigr)\exp(A_1+B_1)-1
}{t}.
\end{align*}
Since \begin{align*}y(1)+A_1+B_1\lesssim\mathcal E_N(0),\end{align*} 
and \(\mathcal{E}_{N}(0)\) is sufficiently small, one has
\begin{align*}
y(t)
\lesssim
t^{-1}\mathcal E_N(0),
\qquad
t\geq1.
\end{align*}
For \(0\leq t\leq1\), the uniform energy estimate \eqref{eq:positive-uniform} gives the same bound
with \(t^{-1}\) replaced by \((1+t)^{-1}\). Hence,
\begin{align}\label{eq:positive-decay}
\mathcal E_N^{(1)}(t)
\leq
C(1+t)^{-1}\mathcal E_N(0),
\qquad
t\geq0.
\end{align}
It follows from \eqref{eq:positive-equivalence} that
\begin{align}\label{eq:positive-L2-decay}
&\sum_{1\leq|\alpha|\leq N-1}
\left\{
\|\partial_x^\alpha f(t)\|_{L_{x,v}^2}
+
\|\partial_x^\alpha u(t)\|_{L^2}
\right\}\leq
C(1+t)^{-\frac{1}{2}}
\left(
\|f_0\|_{L_v^2(H^N)}
+
\|u_0\|_{H^N}
\right).
\end{align}
Thus, the proof of \eqref{TA3a} is completed.

Finally, by applying the first homogeneous Sobolev inequality in Lemma \ref{L2.1}, we can deduce from \eqref{eq:positive-L2-decay} that
\begin{align*}
&\sum_{|\alpha|\leq N-3}
\Big\{
\|\partial_x^\alpha f(t)\|_{L^\infty(L_v^2)}
+
\|\partial_x^\alpha u(t)\|_{L^\infty}
\Big\}\leq
C(1+t)^{-\frac{1}{2}}
\left(
\|f_0\|_{L_v^2(H^N)}
+
\|u_0\|_{H^N}
\right),
\end{align*}
which proves \eqref{TA3} and completes the proof of
Theorem~\ref{Th1}.
\end{proof}

\section{Zero-order \(L^2\) decay of the dissipative variables}

This section is devoted to the proof of Theorem~\ref{Th2}. Since the
common particle-fluid momentum remains undamped at zero spatial
frequency, no uniform algebraic decay rate is available for the complete
zero-order energy. The algebraic convergence established in
\cite{CDM-KRM-2011}, for instance, relies on additional assumptions on
the initial data. Here, without any \(L^1\) or other low-frequency
condition, we show that the directly dissipative variables
$u-J(f)$ and $\{\mathbf I-\mathbf P_0\}f$
decay in the \(L^2\) norm at the rate \((1+t)^{-1/2}\). This result identifies a   zero-order relaxation mechanism that is not visible at the level of the complete energy.

\begin{proof}[Proof of Theorem~\ref{Th2}]
Following \cite[Section 2]{CDM-KRM-2011}, for a velocity function
\(\phi=\phi(x,v)\), we define the moment functional
\(\Gamma_{ij}(\phi)\) by
\begin{align*}
\Gamma_{ij}(\phi)
:=
\big\langle
\phi,(v_iv_j-1)\sqrt M
\big\rangle_v,
\qquad 1\leq i,j\leq3,
\end{align*}
Taking the \(v\sqrt M\)-moment of \eqref{A1}$_3$, we obtain the particle
momentum equation
\begin{align}\label{eq:sec5-J-equation}
\partial_tJ(f)+J(f)+\nabla_xa(f)
+\operatorname{div}_x
\Gamma(\{\mathbf I-\mathbf P_0\}f)
=
u+a(f)u.
\end{align}
Applying the Leray projection \(\mathbb P\) to
\eqref{A1}$_1$ and using \(\mathbb Pu=u\), we obtain
\begin{align}\label{eq:sec5-projected-fluid}
\partial_tu+u
=
\mathbb PJ(f)
-\mathbb P(u\cdot\nabla_xu)
-\mathbb P(a(f)u).
\end{align}
Subtracting \eqref{eq:sec5-J-equation} from
\eqref{eq:sec5-projected-fluid} yields
\begin{align}\label{eq:sec5-relative-momentum}
&\partial_t\{u-J(f)\}
+(\mathbf I+\mathbb P)\{u-J(f)\}\nonumber\\
&\quad=
\nabla_xa(f)
+\operatorname{div}_x
\Gamma\big(\{\mathbf I-\mathbf P_0\}f\big)
-\mathbb P(u\cdot\nabla_xu)
-\mathbb P(a(f)u)-a(f)u.
\end{align}

Since \(\mathbb P\) is an orthogonal projection on \(L^2\), it follows
that
\begin{align*}
&\big(
(\mathbf I+\mathbb P)\{u-J(f)\},
u-J(f)
\big)_x=
\|u-J(f)\|_{L^2}^2
+
\big\|\mathbb P\{u-J(f)\}\big\|_{L^2}^2.
\end{align*}
Moreover, the Gaussian decay of \(\sqrt M\) absorbs the polynomial
velocity weight in the definition of \(\Gamma\), and hence
\begin{align*}
\left\|
\operatorname{div}_x
\Gamma\big(\{\mathbf I-\mathbf P_0\}f\big)
\right\|_{L^2}
\leq
C\left\|
\nabla_x\{\mathbf I-\mathbf P_0\}f
\right\|_{L_{x,v}^2}.
\end{align*}
Taking the \(L^2\) inner product of
\eqref{eq:sec5-relative-momentum} with \(u-J(f)\), and then applying
Young's inequality, we obtain
\begin{align}\label{eq:sec5-relative-energy}
&\frac{{\rm d}}{{\rm d}t}
\|u-J(f)\|_{L^2}^2
+
\|u-J(f)\|_{L^2}^2\nonumber\\
&\quad\leq
C\big\{
\|\nabla_xa(f)\|_{L^2}^2
+
\left\|
\nabla_x\{\mathbf I-\mathbf P_0\}f
\right\|_{L_{x,v}^2}^2
\big\}
+
C\big\{
\|u\cdot\nabla_xu\|_{L^2}^2
+
\|a(f)u\|_{L^2}^2
\big\}.
\end{align}

We next estimate the four-moment remainder. Since
\(\mathbf P_0\) commutes with \(\mathcal L\), applying
\(\{\mathbf I-\mathbf P_0\}\) to \eqref{A1}$_3$ yields
\begin{align}\label{eq:sec5-microscopic-equation}
&\partial_t\{\mathbf I-\mathbf P_0\}f
+
\{\mathbf I-\mathbf P_0\}(v\cdot\nabla_xf)
-
\mathcal L\{\mathbf I-\mathbf P_0\}f
=
\{\mathbf I-\mathbf P_0\}
\big(
-u\cdot\nabla_vf+\frac12(u\cdot v)f
\big).
\end{align}
By utilizing the decomposition \(f = a(f)\sqrt{M}+J(f)\cdot v\sqrt{M}+\{\mathbf{I}-\mathbf{P}_0\}f\) and performing integration by parts with respect to \(x\) and \(v\), we can conclude that, for any sufficiently small \(\epsilon>0\), 
\begin{align*}
 \left|
\big(
\{\mathbf I-\mathbf P_0\}(v\cdot\nabla_xf),
\{\mathbf I-\mathbf P_0\}f
\big)_{x,v}
\right|
 \leq
\epsilon
\big\|\{\mathbf I-\mathbf P_0\}f\big\|_\nu^2
+
C_{\epsilon}\|\nabla_xJ(f)\|_{L ^2}^2,
\end{align*}
and
\begin{align*}
&\big|
\big(
-u\cdot\nabla_vf+\frac12(u\cdot v)f,
\{\mathbf I-\mathbf P_0\}f
\big)_{x,v}
\big|\leq
\epsilon
\big\|\{\mathbf I-\mathbf P_0\}f\big\|_\nu^2
+
C_\epsilon\|u\|_{L ^\infty}^2\|J(f)\|_{L ^2}^2
+
C\|u\|_{L ^\infty}
\big\|\{\mathbf I-\mathbf P_0\}f\big\|_\nu^2.
\end{align*}
The uniform smallness from \eqref{TA2} allows the last term to
be absorbed into the Fokker-Planck dissipation. Taking the
\(L_{x,v}^2\) inner product of
\eqref{eq:sec5-microscopic-equation} with
\(\{\mathbf I-\mathbf P_0\}f\), and using the coercivity of
\(\mathcal L\), we obtain
\begin{align}\label{eq:sec5-microscopic-energy}
\frac{{\rm d}}{{\rm d}t}
\big\|\{\mathbf I-\mathbf P_0\}f\big\|_{L_{x,v}^2}^2
+
\lambda_6
\big\|\{\mathbf I-\mathbf P_0\}f\big\|_\nu^2\leq
C\|\nabla_xJ(f)\|_{L ^2}^2
+
C\|u\|_{L ^\infty}^2\|J(f)\|_{L ^2}^2
\end{align}
for some constant \(\lambda_6>0\).

Combining \eqref{eq:sec5-relative-energy} with
\eqref{eq:sec5-microscopic-energy},  we arrive at
\begin{align*}
&\frac{{\rm d}}{{\rm d}t}
\Big(
\|u-J(f)\|_{L ^2}^2
+
\big\|\{\mathbf I-\mathbf P_0\}f\big\|_{L_{x,v}^2}^2
\Big)
+
\lambda_7
\Big(
\|u-J(f)\|_{L ^2}^2
+
\big\|\{\mathbf I-\mathbf P_0\}f\big\|_{L_{x,v}^2}^2
\Big)
\nonumber\\
&\quad\leq
C\Big(
\|\nabla_xa(f)\|_{L ^2}^2
+
\|\nabla_xJ(f)\|_{L ^2}^2
+
\big\|\nabla_x\{\mathbf I-\mathbf P_0\}f\big\|_{L_{x,v}^2}^2
\Big)
\nonumber\\
&\quad\quad
+
C\|u\|_{L^\infty}^2
\big(
\|\nabla_xu\|_{L ^2}^2
+
\|a(f)\|_{L ^2}^2
+
\|J(f)\|_{L ^2}^2
\big),
\end{align*}
for some constant $\lambda_{7}>0$.

By using the positive-order decay \eqref{TA3a}  and the pointwise
estimate \eqref{TA3}, we have
\begin{align*}
\|\nabla_xa(f)(t)\|_{L^2}^2
+\|\nabla_xJ(f)(t)\|_{L^2}^2
+\big\|\nabla_x\{\mathbf I-\mathbf P_0\}f(t)\big\|_{L_{x,v}^2}^2
+\|\nabla_xu(t)\|_{L^2}^2
+\|u(t)\|_{L^\infty}^2
\leq
C(1+t)^{-1}\mathcal E_N(0).
\end{align*}
Moreover, the uniform estimate \eqref{TA2} implies that
\begin{align*}
\|a(f)(t)\|_{L^2}^2+\|J(f)(t)\|_{L^2}^2
\leq C\mathcal E_N(0),
\end{align*}
which, together with the smallness of
\(\mathcal E_N(0)\), yields
\begin{align}\label{eq:sec5-forcing-decay}
&\frac{{\rm d}}{{\rm d}t}
\Big(
\|u-J(f)\|_{L^2}^2
+
\big\|\{\mathbf I-\mathbf P_0\}f\big\|_{L_{x,v}^2}^2
\Big)
+\lambda_7
\Big(
\|u-J(f)\|_{L^2}^2
+
\big\|\{\mathbf I-\mathbf P_0\}f\big\|_{L_{x,v}^2}^2
\Big)\nonumber\\
&\quad\leq
C(1+t)^{-1}\mathcal E_N(0).
\end{align}
Applying Gr\"{o}nwall's inequality, we obtain
\begin{align}\label{G5.9}
\|u(t)-J(f)(t)\|_{L^2}^2
+
\big\|\{\mathbf I-\mathbf P_0\}f(t)\big\|_{L_{x,v}^2}^2
\leq&\,
C{\rm e}^{-\lambda_7t}\mathcal E_N(0)
+
C\mathcal E_N(0)
\int_0^t
{\rm e}^{-\lambda_7(t-\tau)}
(1+\tau)^{-1}\,{\rm d}\tau\nonumber\\
\leq&\,
C(1+t)^{-1}\mathcal E_N(0).
\end{align}
Taking square roots in \eqref{G5.9} proves
\eqref{TA4} and completes the proof of Theorem~\ref{Th2}.
\end{proof}

\bigskip 
\noindent{\bf Acknowledgements:} 
The author sincerely thanks Professor Renjun Duan for his generous assistance and for many valuable discussions on kinetic theory this year.

\vspace{2mm}

\noindent\textbf{Conflict of interest.} The authors do not have any possible conflicts of interest.

\vspace{2mm}

\noindent\textbf{Data availability statement.}
 Data sharing is not applicable to this article as no data sets were generated or analyzed during the current study.

\bibliographystyle{plain}

\end{document}